\documentclass[11pt,reqno]{amsart}
\usepackage{amssymb,mathrsfs,graphicx,subfigure, enumerate}
\usepackage{amsmath,amsfonts,amssymb,amscd,amsthm,bbm}
\usepackage{graphicx,colortbl}
\usepackage{extpfeil}

\def\bP {\mathbf{P}}

\def\bS {\mathbf{S}}
\def\bT {\mathbf{T}}

\def\fH {\mathfrak{H}}
\def\fS {\mathfrak{S}}

\def\cD {\mathcal{D}}

\def\cP {\mathcal{P}}

\def\cS {\mathcal{S}}

\def\cV {\mathcal{V}}

\def\a {{\alpha}}

\def\de {{\delta}}

\def\ka {{\kappa}}

\def\si {{\sigma}}
\def\Si {{\Sigma}}

\def\om {{\omega}}

\def\d {{\partial}}

\def\Dlt {{\Delta}}

\def\rstr {{\big |}}
\def\indc {{\bf 1}}

\def\la {\langle}
\def\ra {\rangle}
\def \La {\bigg\langle}
\def \Ra {\bigg\rangle}

\newcommand{\Div}{\operatorname{div}}

\newcommand{\RE}{\operatorname{Re}}

\newcommand{\bbr}{\mathbb R}
\newcommand{\bbc}{\mathbb C}

\newcommand{\bbz}{\mathbb Z}

\newcommand{\bbt} {\mathbb T}

\newcommand{\ba}{\begin{aligned}}
\newcommand{\ea}{\end{aligned}}

\newcommand{\be}{\begin{equation}}
\newcommand{\ee}{\end{equation}}

\newcommand{\lb}{\label}

\newtheorem{Thm}{Theorem}[section]
\newtheorem{Lem}{Lemma}[section]

\newtheorem{Prop}{Proposition}[section]
\newtheorem{Rmk}{Remark}[section]
\newtheorem{Def}{Definition}[section]

\numberwithin{equation}{section}

\begin{document}

\title[The mean-field limit of the Schr\"{o}dinger-Lohe model]{The mean-field limit of the Schr\"{o}dinger-Lohe model\\ and emergent dynamics}

\author[F. Golse]{Fran\c cois Golse}
\address[F. Golse]{CMLS, \'Ecole polytechnique, IP Paris \& CNRS, 91128 Palaiseau Cedex, France}
\email{francois.golse@polytechnique.edu}

\author[S.-Y. Ha]{Seung-Yeal Ha}
\address[S.-Y. Ha]{
Department of Mathematical Sciences and Research Institute of Mathematics, \newline
Seoul National University, Seoul, 08826, Republic of Korea}
\email{syha@snu.ac.kr}

\thanks{\textbf{Acknowledgment.} The work of S.-Y. Ha is supported by the National Research Foundation of Korea (NRF-RS2025-00514472)}

\begin{abstract}
The Schr\"odinger-Lohe (SL) model is a coupled system of nonlinear Schr\"odin- ger equations describing the temporal-spatial evolution of the component wave functions, and it corresponds to the infinite-dimensional counterpart of the Lohe matrix model for quantum synchronization. In this paper, we study a rigorous mean-field limit of the SL model and  provide a quantitative estimate on the fluctuations between one-marginal distribution and one particle distribution in 2-Wasserstein distance. For the derived kinetic SL equation, we present sufficient conditions leading to the complete and practical synchronizations.
\end{abstract}



\date{\today}

\maketitle



\section{Introduction} \label{sec:1}
\setcounter{equation}{0}

Collective behaviors of a complex system are often observed in physical and biological systems, e.g., flashing of fireflies \cite{B-B}, synchronization of weakly coupled oscillators \cite{A-B, D-B1, H-K-P-Z, St}, 
heart beat regulation by pacemaker cells \cite{Pe, P-R}, arrays of Josephson junctions \cite{S-M, W-S}, etc.  However, in spite of their omnipresence,  model-based study of these phenomena was first begun by A. Winfree \cite{Wi2, Wi1} 
and Y. Kuramoto \cite{Ku1, Ku2} only a half century ago. After Winfree and Kuramoto's seminal works, several mathematical models were proposed in literature. Among them, we are mainly interested in the so-called Schr\"odinger-Lohe 
(henceforth SL) model introduced by M. Lohe in \cite{Lo-1} as a toy model for quantum synchronization. To set up the stage, we begin with the Kuramoto model for weakly coupled limit-cycle oscillators. 

Let $\theta_k = \theta_k(t)$ be the phase of the $k$-th Kuramoto oscillator. Then, its temporal dynamics is governed by the Cauchy problem for the Kuramoto model:
\begin{equation}
\begin{cases} \label{KM}
\displaystyle {\dot \theta}_k = \nu_k + \frac{\kappa}{N} \sum_{l=1}^{N} \sin(\theta_l - \theta_k), \quad t > 0, \\
\displaystyle \theta_k \Big|_{t = 0} = \theta_k^0, \quad  k \in [N]:=\{1, \cdots, N \}, 
\end{cases}
\end{equation}
where $\nu_k$ and $\kappa$ denote the natural (intrinsic) frequency of the $k$-th oscillator and the coupling strength, respectively. The emergent dynamics of \eqref{KM} has been extensively investigated in literature \cite{C-H-J-K, C-S, 
D-X, D-B1, D-B, H-K-R, H-N-P2,M-S1, M-S2, M-S3,V-M1, V-M2}. Recently, several  higher-dimensional and infinite-dimensional counterparts of \eqref{KM} have been proposed in control and statistical physics communities, to name a few,  
swarming models on spheres \cite{C-G-O, C-H5, Lo-2, Ma1,Ol}, synchronization models on Riemannian manifolds \cite{De, G-H, H-K-R1, H-K-R2, H-R, M-T-G, R-L-W}, quantum mechanical models for synchronization 
\cite{C-C-H,G-G,G-C, H-L-G,X-T, Z-S} and so on. 

Next, we briefly discuss the infinite-dimensional counterpart of \eqref{KM}, namely ``{\it the SL model} ''. Consider a complete quantum network  \cite{Ki} consisting of $N$ nodes (vertices), and assume that quantum SL oscillators are 
located at nodes, and interact each other through laser pulses. Let $\psi_k = \psi_k(t, x)\in\bbc$ be the wave function of the $k$-th SL oscillator with unit mass $m = 1$ at position $x \in \bbt^d := \bbr^d /\bbz^d$, at time $t \in \bbr$. Let the pairing 
$\langle \cdot~ |~\cdot \rangle$ and $\|\cdot\|_\fH$ denote the standard inner product in $\fH := L^2(\bbt^d)$, and the corresponding norm in the Hilbert space $\fH$: for $\psi_l, \psi_k \in \fH$, 
\[ 
\langle \psi_l |\psi_k \rangle := \int_{\bbt^d}  \psi_l(x) \overline{\psi_k(x)} dx \quad \mbox{and} \quad \| \psi_k \|_{\fH}  := \sqrt{\langle \psi_k | \psi_k \rangle}. 
\]
Let $C(\bbt^d)$ be the space of complex-valued continuous functions on the $d$-dimensional tori $\bbt^d$, and let $\bS\fH:=\{\psi\in\fH~:~\|\psi\|_\fH=1\}$ denote the unit sphere in the Hilbert space $\fH$. 
For $V \in C(\bbt^d)$, we set 
\[ 
||V||_{L^{\infty}} = \|V \|_{L^{\infty}(\bbt^d)} :=  \max_{x \in \bbt^d} |V(x)|.
\]
Then, the dynamics of $\psi_k$ is governed by the following Cauchy problem for the SL model: 
\begin{equation} \label{SL}
\begin{cases}
\displaystyle {\mathrm i}\hbar\d_t\psi_k=(H_k+\om\star_x|\psi_k|^2)\psi_k+\frac{{\mathrm i}\hbar\kappa}{2N}\sum_{l=1}^N(\psi_l-\la\psi_k| \psi_l\ra\psi_k)\,, \quad t > 0,~~x \in \bbt^d,
\\
\displaystyle \psi_k\rstr_{t=0}=\psi_k^{0}\,, \quad  ||\psi_k^{0}||_{\fH} = 1,  \quad k \in [N],
\end{cases}
\end{equation}
where  $\kappa$ and $\hbar$ are the nonnegative coupling strength and the reduced Planck constant respectively, while $H_k$ is the Hamiltonian of the $k$-th quantum system:
\begin{equation} \label{A-0-0}
H_k := -\tfrac12\hbar^2\Dlt_x + V_k.
\end{equation}
We also assume that convolution kernel $\omega$ is uniformly bounded throughout the paper. 

In the absence of the Hartree term $\om\star_x|\psi_k|^2$, the system \eqref{SL} was first introduced in \cite{Lo-1} as the infinite-dimensional counterpart of the Lohe matrix model \cite{Lo-2} on the unitary group, and its synchronization 
properties have been exensively studied in a series of papers \cite{A-M, B-H-K-T, C-C-H, C-H5, H-H, H-H-K1, H-H-K2}. Moreover, the nonlinear term on the right-hand side of \eqref{SL} is sub-linear so that 
the well-posedness of \eqref{SL} can be derived via the standard arguments for nonlinear Schr\"{o}dinger-type equations (see Appendix A in \cite{H-H} for details). We also refer to  \cite{G-G, G-C, Ki} for the potential applications 
of this model in quantum computing and quantum information.  In this paper, we address the following two questions:
\vspace{0.1cm}
\begin{itemize}
\item
(Q1):~What is the kinetic equation which can be obtained from \eqref{SL} via the mean-field limit as $N \to \infty$? 
\vspace{0.2cm}
\item
(Q2):~Under what conditions on initial data and system parameters, can we show synchronization for the derived kinetic equation?
\end{itemize}
Consider the special case:
\begin{equation*} \label{A-0}
\omega = 0, \quad H_k = \hbar \nu_{k}~:~\mbox{constant} \qquad \psi_k(t, x) = e^{-{\mathrm i} \theta_{k}}, \quad (t, x) \in \bbr_+ (:= (0, \infty)) \times  \bbt^d.
\end{equation*} 
In this case, system $\eqref{SL}_1$ reduces to the Kuramoto model \eqref{KM}. This is why we call the SL model the infinite-dimensional counterpart of the Kuramoto model. Let $U(t)$ and $\phi_k(t)$ be the propagator for the free 
Schr\"{o}dinger equation with $\hbar = 1$ and the transformed wave functions, respectively:
\[ 
U(t) \psi^{0} := e^{\frac{{\mathrm i} t  \Delta}{2}} \psi^{0} \quad \mbox{and} \quad \phi_k(t):=U(\hbar t)^*\psi_k(t), \quad k \in [N]. 
\]
Then, the pair $(\phi_k, V_k)$ satisfies the transformed SL model (see Section \ref{sec:2.3} for details):
\begin{equation} \label{A-1}
\begin{cases}
\displaystyle \d_t\phi_k =\tfrac1{i\hbar}U(\hbar t)^*(V_k+\om\star_x|U(\hbar t)\phi_k|^2)U(\hbar t)\phi_k\!+\!\frac{\kappa}{2N}\sum_{l=1}^N \Big(\phi_l\!-\!\la\phi_k|\phi_l\ra\phi_k \Big ), \\
\displaystyle \d_t V_k = 0, \quad t > 0, \quad x \in \bbt^d, \quad k \in [N], \\
\displaystyle \phi_k\rstr_{t=0}=\psi_k^{0}\,, \quad V_k \rstr_{t=0}= V^0_k.
\end{cases}
\end{equation}
Note that the (unbounded) operator $\Delta$ does not appear in \eqref{A-1} any more. \newline

The main results of this paper are two-fold. First, we present a rigorous derivation of the mean-field, kinetic SL equation for the transformed SL model \eqref{A-1}. The exact form of the mean-field equation can be identified using 
the standard BBGKY hierarchy  argument from statistical physics, assuming the existence of weak limits for the marginal probability density functions and molecular chaos assumption (see Section \ref{sec:3}). Motivated by recent 
progress \cite{B-G-M, G-P, G-M-P} on the mean-field limit for Schr\"odinger type systems, we provide a quantitative fluctuation estimate using the $2$-Wasserstein metric (see Section \ref{sec:4}). More precisely,  let $X_N$ and 
$F_N$ be the vector field associated with system \eqref{A-1} and the $N$-oscillator distribution function on the $N$-oscillator phase space $\bS\fH^N\times C(\bbt^d)^N$, respectively. 

Then, $N$-oscillator distribution function $F_N$ satisfies the Cauchy problem for the Liouville equation: 
\begin{align}
\left\{
\begin{aligned}  \label{A-1-1}
& \d_tF_N+\Div_{(\phi, V)}(F_NX_N)=0\,, \quad (t, \phi, V) \in \bbr_+ \times (\bS\fH)^N \times C(\bbt)^N, 
\\
& F_N\rstr_{t=0}=F_N^{0}\in\cP((\bS\fH)^N\times C(\bbt^d)^N)\,,
\end{aligned}
\right.
\end{align}
where $\cP(Z)$ designates the set of Borel probability measures on the complete  separable metric space $Z$.

On the other hand, let $f = f(t, d\phi dV)$ be the one-oscillator ``distribution function'' (probability measure, in fact) for the infinite ensemble of wave functions on the (one-oscillator) phase space $\bS\fH \times C(\bbt^d)$ at time $t$. Then, the 
kinetic SL equation for $f$ reads as follows (see Section \ref{sec:3}):
\begin{align}
\begin{aligned} \label{A-2}
& \d_tf +\Div_\phi\left(f\left(\tfrac1{i\hbar}U(\hbar t)^*(V+\om\star_x|U(\hbar t)\phi|^2)U(\hbar t)\phi\right)\right) \\
& \hspace{2cm} + \tfrac12\kappa \Div_\phi\left(f\int_{\bS\fH\times C(\bT^d)}(\tilde\phi-\la\phi|\tilde\phi\ra\phi)f(t,d\tilde\phi d\tilde V)\right)=0.
\end{aligned}
\end{align}
Let $f^{0} \in\cP_2(\bS\fH\times C(\bbt^d))$ (the set of Borel probability measures on $\bS\fH\times C(\bbt^d)$ with finite 2nd moment), and let $f$ be the solution to \eqref{A-2} with initial datum $f^{0}$. On the other hand,  let 
$F_N$ be the solution of the Liouville equation \eqref{A-1-1} with the initial datum $F_N^{0}=(f^{0})^{\otimes N}$. Then, the $2$-Wasserstein distance between $f$ and the first marginal $F_{N:1}$ of $F_N$ satisfies the following estimate:~for $T \in (0, \infty)$,
\[  
W_2(F_{N:1}(t),f(t)) \leq \frac{C(T)}{\sqrt{N}} \to 0, \quad \mbox{as $N \to \infty, \quad$ for all $t \in (0, T),$} 
\] 
(see Theorem \ref{T4.1} in Section \ref{sec:4}). This establishes the mean-field limit and specifies the convergence rate for that limit. 

Second, we deal with the emergent dynamics of \eqref{A-2}. For this, we consider an ensemble of zero identical Hamiltonians and Hartree force, and pick the density function $f$ of the form (mono-kinetic ansatz):
\[ 
f(t, d\phi dV) = \rho(t,d\phi)\delta_0(dV).
\]
Then, the local mass density $\rho \equiv \rho(t, d\phi) := \int_{C(\bT^d)}  f(t, d\phi dV)$ satisfies the nonlocal continuity equation:
\begin{equation} \label{A-3}
\d_t \rho + \frac{\kappa}{2} \Div_\phi\left(\rho \int_{\bS\fH} K(\phi, \tilde \phi) \rho(t,d\tilde \phi)\right)=0, \quad (t, \phi) \in \bbr_+ \times \bS\fH,
\end{equation}
where the kernel $K$ is given by the following relation:
\[  K(\phi, {\tilde \phi}) := \tilde\phi-\la\phi|\tilde\phi\ra\phi. \]
For the synchronization estimate of \eqref{A-3}, we consider the diameter of the $\phi$-support of $\rho$:
\[   
{\mathcal D}[\rho(t)] := \sup \Big \{ \| \phi - {\tilde \phi} \|_{\fH}:~ \phi, {\tilde \phi} \in \mbox{supp}~\rho(t,\cdot)  \Big \} \leq 2.  
\]
Suppose that the coupling strength $\kappa$ and the initial datum $\rho^0$ satisfy
\[  
\kappa > 0, \quad {\mathcal D}[\rho^0] <  \frac{1}{2}, \quad \int_{\bS\fH} \rho^0(d\phi)= 1.
\]
Then one can show that ${\mathcal D}[\rho(t)]$ tends to zero exponentially fast (complete synchronization, Theorem \ref{T5.1}):
\[    
{\mathcal D}[\rho(t)]  \leq  \frac{\mathcal D[\rho^0]}{(1-2\mathcal D[\rho^0])e^{\frac{\kappa}{2} t}+2\mathcal D[\rho^0]}, \quad t \geq 0. 
\]
On the other hand, for an ensemble of nonidentical Hamiltonians, we assume that the coupling strength, one-body potentials and initial datum satisfy 
\[
\begin{cases}
& P_\infty :=  \sup \Big \{ \|V \|_{L^{\infty}} \Big |~(\phi, V)  \in \mbox{supp}(f^0(\cdot)) \Big \} < \infty, \\
& \alpha := \tfrac1\hbar\sup \Big \{ \|V_1 - V_2 \|_{L^{\infty}} \Big |~(\phi_1, V_1)~\mbox{and}~(\phi_2, V_2) \in \mbox{supp}(f(t, \cdot)) \Big \} < \infty,  \\
& {\mathcal D}[f_\kappa^0] <\xi_2(\a/\ka), \quad \int_{\bS\fH} \rho_k^0(d\phi) = 1, 
\end{cases}
\]
where $\xi_2$ is the larger root of a quadratic equation defined in \eqref{NNN-1}, and ${\mathcal D}[f_\kappa]$ is similarly defined as $\phi$-support of $\rho_{f_\kappa}$. Then, we can derive the practical synchronization (see Theorem \ref{T6.1}):
\[
\lim_{\ka\to \infty}\varlimsup_{t\to \infty} {\mathcal D}[f_\kappa(t)]=0.
\]

\vspace{0.5cm}

The rest of this paper is organized as follows. In Section \ref{sec:2}, we study a priori estimates of the SL model \eqref{SL} such as the $L^2$-conservation law, complete and practical synchronizations for the SL model, together with a reformulation \eqref{A-1} of the SL model. In Section \ref{sec:3}, we study the Liouville equation for the $N$-oscillator distribution and the derivation of the kinetic SL equation from the Liouville equation via the BBGKY hierarchy argument, and then discuss the existence and uniqueness theory for the kinetic SL equation \eqref{A-2}. In Section \ref{sec:4}, we provide a quantitative estimate on the $2$-Wasserstein distance between the single-particle distribution that is the solution of the kinetic SL equation and the first marginal of the $N$-particle distribution. In Section \ref{sec:5} and Section \ref{sec:6}, we study two types of emergent dynamics (complete synchronization v.s. practical synchronization) for the kinetic SL equation.  Finally, Section \ref{sec:7} is devoted to the brief summary of our main results and some remaining issues for a future work. In Appendix \ref{S-CL}, we recall the basic Cauchy-Lipschitz theory for systems of abstract  ordinary differential equations on Banach spaces, while in Appendix \ref{SLiouv}, we present an existence theory of a Lagrangian weak solution for the Liouville equation on the $N$-oscillator phase space.


\section{Preliminaries} \label{sec:2}


In this section, we study the emergent properties of the SL model \eqref{SL}. First, we recall the $L^2$-conservation of each wave function $\psi_k$ and the definitions of complete and of practical synchronizations
for \eqref{SL}. In Section \ref{sec:2.1} and Section \ref{sec:2.2}, we present formal computations and estimates in order to keep the exposition as clear as possible, and then in Section \ref{sec:2.3}, we propose one easy way of making
the estimates obtained in Section  \ref{sec:2.1} and Section \ref{sec:2.2} be rigorous. Finally, we shall conclude with a statement on the practical synchronization for \eqref{SL} in Theorem \ref{T-PractiSynchroSL}.

\subsection{$L^2$-conservation law} \label{sec:2.1}
In this subsection, we show that the $L^2$-norm is conserved along the trajectories of \eqref{SL} as follows.
\begin{Lem}\label{L2.1}
\emph{($L^2$-conservation)}
 For $T \in (0, \infty]$, let $\Psi := \{ \psi_k \}$ be a smooth solution to \eqref{SL} in the time-interval $[0, T)$. Then, one has
\[  
|| \psi_k(t) ||_{\fH}= 1, \quad \text{ for each }t \in [0, T)\text{ and each } k \in [N]. 
\]
\end{Lem}

\begin{proof} We use the Hermitian property of $H_k$ (i.e. $H^*_k=H_k$) to write
\begin{align}
\begin{aligned} \label{B-1}
\d_t\psi_k &= -\frac{{\mathrm i}}{\hbar} (H_k+\om\star_x|\psi_k|^2)\psi_k+\frac{\kappa}{2N}\sum_{l=1}^N(\psi_l-\la\psi_k| \psi_l\ra\,\psi_k), 
\\
\d_t {\overline \psi}_k &= +\frac{{\mathrm i}}{\hbar} (H_k+\om\star_x|\psi_k|^2){\overline \psi}_k+\frac{\kappa}{2N}\sum_{l=1}^N({\overline \psi}_l- \overline{\la\psi_k | \psi_l \ra}\,{\overline  \psi}_k).
\end{aligned}
\end{align}
Multiplying both sides of the equality $\eqref{B-1}_1$ by ${\overline \psi}_k$ and both sides of the equality $\eqref{B-1}_2$ by $\psi_k$, and adding the resulting expressions yield
\[
\d_t |\psi_k|^2 = \frac{\kappa}{2N} \sum_{l=1}^{N} \Big[  \psi_l {\overline \psi}_k + \psi_k {\overline \psi}_l -  \Big(   \la \psi_k | \psi_l \ra + \overline{\la \psi_k| \psi_l \ra}  \Big)  |\psi_k|^2                 \Big].   
\]
We use $\displaystyle \psi_c := \frac{1}{N} \sum_{l=1}^{N}  \psi_l$ and integrate both sides of the above relation over $\bbt^d$ to obtain
\begin{equation*}
 \frac{d}{dt} \Big( 1 - ||\psi_k||_{\fH}^2 \Big) =  -\kappa\mbox{Re} \la \psi_k | \psi_c \ra  \Big( 1 - ||\psi_k||_{\fH}^2 \Big).
\end{equation*}
This implies
\begin{equation} \label{B-2}
1 - ||\psi_k(t)||_{\fH}^2 = (1 - ||\psi^0_k||_{\fH}^2) \exp \Big(-\int_0^t \kappa \mbox{Re} \la \psi_k(s) | \psi_c(s) \ra ds \Big). 
\end{equation}
Since $||\psi^0_k|| = 1$, the relation \eqref{B-2} yields
\[ 
1 - ||\psi_k(t)||_{\fH}^2 = 0, \quad \text{ for}~t \in [0, T)\text{ and}~k  \in [N]. 
\]
\end{proof}

\subsection{Emergent dynamics estimates} \label{sec:2.2}
In this subsection, we provide emergent estimates for \eqref{SL}. For this, we first recall the concepts of complete and practical synchronizations in the following definition.
\begin{Def}\label{D2.1}
\emph{\cite{C-C-H}}
Let $\Psi := \{ \psi_k \}$ be a global smooth solution to \eqref{SL}. 
\begin{enumerate}
\item
The ensemble $\Psi$ exhibits {\em complete synchronization}, if the following estimate holds:
\[ 
\lim_{t \to \infty} \max_{1\le k,l\le N} \| \psi_l(t) - \psi_k(t) \|_{\fH} = 0. 
\]
\item
The ensemble $\Psi$ exhibits {\em practical synchronization},  if the following estimate holds:
\[ 
\lim_{\kappa \to \infty} \limsup_{t \to \infty} \max_{1\le k,l\le N} \| \psi_l(t) - \psi_k(t) \|_{\fH} = 0. 
\]
\end{enumerate}
\end{Def}
Next, we estimate the temporal evolution of $L^2$-distance $\|\psi_k(t) - \psi_j(t) \|_{\fH}$. It follows from \eqref{SL} - \eqref{A-0-0} that
\begin{align*}
\begin{aligned}
\d_t\psi_ik&= \frac{{\mathrm i} \hbar }{2} \Delta\psi_k-\frac{{\mathrm i}}{\hbar} V_k \psi_k  -\frac{{\mathrm i}}{\hbar} ( \om\star_x|\psi_k|^2 )  \psi_k +\frac{\kappa}{2N}\sum_{l=1}^N(\psi_l-\la\psi_k |\psi_l\ra\psi_k), 
\\
\d_t\psi_j &= \frac{{\mathrm i} \hbar }{2} \Delta\psi_j -\frac{{\mathrm i}}{\hbar} V_j \psi_j -\frac{{\mathrm i}}{\hbar} ( \om\star_x|\psi_j|^2 )  \psi_j +\frac{\kappa}{2N}\sum_{l=1}^N(\psi_l-\la\psi_j| \psi_l\ra\psi_j).
\end{aligned}
\end{align*}
Substracting both sides of these equalities implies
\begin{align}
\begin{aligned} \label{B-5}
\partial_t(\psi_k-\psi_j) &= \frac{{\mathrm i} \hbar}{2}\Delta(\psi_k-\psi_j) -\frac{{\mathrm i}}{\hbar} V_k\psi_k
+ \frac{{\mathrm i}}{\hbar} V_j \psi_j -\frac{{\mathrm i}}{\hbar} ( \omega \star_x |\psi_k|^2 )  \psi_k   \\
&\hspace{0.2cm} + \frac{{\mathrm i}}{\hbar} ( \omega \star_x |\psi_j|^2 )  \psi_j  + \frac{\kappa}{2N}\sum_{l=1}^N\Big(\langle\psi_j|\psi_l\rangle\psi_j -\langle\psi_k|\psi_l \rangle\psi_k \Big).
\end{aligned}
\end{align}
Next, we multiply both sides of \eqref{B-5} by $\overline{\psi_k-\psi_j}$, and integrate both sides of the resulting equality over $\bbt^d$ to obtain
\begin{align}
\begin{aligned} \label{B-6}
&\langle\partial_t(\psi_k-\psi_j)| \psi_k-\psi_j \rangle 
\\
&\hspace{0.5cm} =\frac{{\mathrm i} \hbar}{2} \langle \Delta(\psi_k-\psi_j)| \psi_k-\psi_j \rangle - \frac{{\mathrm i}}{\hbar} \langle V_k\psi_k| \psi_k-\psi_j \rangle
  + \frac{{\mathrm i}}{\hbar} \langle V_j\psi_j| \psi_k-\psi_j \rangle 
\\
&\hspace{0.7cm} -\frac{{\mathrm i}}{\hbar}  \langle ( \omega \star_x |\psi_k|^2 )  \psi_k | \psi_k - \psi_j \rangle + \frac{{\mathrm i}}{\hbar}  \langle ( \omega \star_x |\psi_j|^2 )  \psi_j| \psi_k - \psi_j \rangle  
\\
& \hspace{0.7cm} + \frac{\kappa}{2N}\sum_{l=1}^N\bigg(\langle \psi_j | \psi_l \rangle\langle \psi_j| \psi_k-\psi_j \rangle -\langle\psi_k| \psi_l \rangle \langle \psi_k| \psi_k-\psi_j \rangle\bigg).
\end{aligned}
\end{align}
Now, we take the complex conjugate of each side of \eqref{B-6} to obtain
\begin{align}
\begin{aligned} \label{B-7}
&\langle \psi_k-\psi_j| \partial_t(\psi_k-\psi_j) \rangle
\\
& \hspace{0.8cm} = -\frac{{\mathrm i} \hbar}{2} \langle \psi_k - \psi_j| \Delta(\psi_k-\psi_j) \rangle + \frac{{\mathrm i}}{\hbar} \langle \psi_k - \psi_j | V_k\psi_k \rangle - \frac{{\mathrm i}}{\hbar} \langle \psi_k - \psi_j | V_j\psi_j \rangle 
\\
& \hspace{1cm} + \frac{{\mathrm i}}{\hbar}  \langle \psi_k - \psi_j | ( \omega \star_x |\psi_k|^2 )  \psi_k \rangle - \frac{{\mathrm i}}{\hbar}  \langle \psi_k - \psi_j | ( \omega \star_x |\psi_j|^2 )  \psi_j  \rangle  
\\
& \hspace{1cm} + \frac{\kappa}{2N}\sum_{l=1}^N\bigg(\langle \psi_l, \psi_j \rangle\langle\psi_k - \psi_j | \psi_j \rangle -\langle\psi_l| \psi_k \rangle \langle  \psi_k - \psi_j| \psi_k \rangle\bigg).
\end{aligned}
\end{align}
Finally, we add both sides of \eqref{B-6} and \eqref{B-7} and find that
\begin{align}
\begin{aligned} \label{B-8}
\begin{split}
&\frac{d}{dt}||\psi_k-\psi_j||_{\fH}^2 =  \frac{{\mathrm i} \hbar}{2} \langle \Delta(\psi_k-\psi_j) | \psi_k-\psi_j \rangle  -\frac{{\mathrm i} \hbar}{2} \langle \psi_k - \psi_j | \Delta(\psi_k-\psi_j) \rangle  
\\
&\hspace{0.5cm} - \frac{{\mathrm i}}{\hbar} \langle V_k\psi_k | \psi_k-\psi_j \rangle + \frac{{\mathrm i}}{\hbar} \langle \psi_k  - \psi_j| V_k\psi_k \rangle
\\
&\hspace{0.5cm} + \frac{{\mathrm i}}{\hbar} \langle V_j\psi_j | \psi_k-\psi_j \rangle - \frac{{\mathrm i}}{\hbar} \langle \psi_k - \psi_j | V_j\psi_j \rangle 
\\
&\hspace{0.5cm} -\frac{{\mathrm i} }{\hbar}  \langle ( \omega \star_x |\psi_k|^2 )  \psi_k | \psi_k - \psi_j \rangle + \frac{{\mathrm i} }{\hbar}  \langle ( \omega \star_x |\psi_j|^2 )  \psi_j | \psi_k - \psi_j \rangle 
\\
&\hspace{0.5cm} + \frac{{\mathrm i}}{\hbar}  \langle \psi_k - \psi_j | ( \omega \star_x |\psi_k|^2 )  \psi_k \rangle - \frac{{\mathrm i}}{\hbar}  \langle \psi_k - \psi_j | ( \omega \star_x |\psi_j|^2 )  \psi_j  \rangle  
\\
&\hspace{0.5cm}+ \frac{\kappa}{2N}\sum_{l=1}^N\bigg(\langle \psi_j| \psi_l \rangle\langle \psi_j| \psi_k-\psi_j \rangle -\langle\psi_k | \psi_l \rangle \langle \psi_k | \psi_k-\psi_j \rangle\bigg) 
\\
&\hspace{0.5cm}+ \frac{\kappa}{2N}\sum_{l=1}^N\bigg(\langle \psi_l | \psi_j \rangle\langle\psi_k - \psi_j | \psi_j \rangle -\langle\psi_l | \psi_k \rangle \langle  \psi_k - \psi_j | \psi_k \rangle\bigg) \\
&\hspace{0.5cm} =: \sum_{k=1}^{7} {\mathcal I}_{1k}. 
\end{split}
\end{aligned}
\end{align}
In the following lemma, we provide the temporal evolution of $\frac{d}{dt}||\psi_k-\psi_j||_{\fH}^2$ by estimating ${\mathcal I}_{1k}$.
\begin{Lem} \label{L2.2}
\emph{\cite{C-C-H}}
For $T \in (0, \infty]$, let $\{ \psi_k \}$ be a smooth solution to \eqref{SL} in the time-interval $[0, T)$. Then, we have
\begin{align}
\begin{aligned}  \label{B-8-1}
& \frac{d}{dt}||\psi_k-\psi_j||_{\fH}^2 \\
& \hspace{0.5cm} \leq-\frac{2}{\hbar}\int_{\bbt^d}(V_k-V_j)\mbox{Im}(\psi_k{\overline\psi}_j)dx-\frac{2}{\hbar}\int_{\bbt^d}\omega\star_x (|\psi_k|^2-|\psi_j|^2)\mbox{Im}(\psi_k{\overline \psi}_j)dx 
\\
&\hspace{0.7cm} +\frac{\kappa}{N}\sum_{k=1}^N \Big[\Big(||\psi_k-\psi_k||_{\fH} +||\psi_j-\psi_k||_{\fH}  - 1\Big)\cdot||\psi_k-\psi_j||_{\fH}^2-|1-\langle \psi_k|\psi_j \rangle|^2 \Big]. 
\end{aligned}
\end{align}
\end{Lem}
\begin{proof} In the sequel, we study each term ${\mathcal I}_{1k}$ for $k \in [7]$ separately.

\vspace{0.2cm}
\noindent
$\bullet$ (Estimate of ${\mathcal I}_{11}$):  Since $\Delta$ is self-adjoint on $L^2(\bbt^d)$, we have
\[
{\mathcal I}_{11} = \frac{{\mathrm i} \hbar}{2} \langle \Delta(\psi_k-\psi_j) | \psi_k-\psi_j \rangle  -\frac{{\mathrm i} \hbar}{2} \langle \psi_k - \psi_j | \Delta(\psi_k-\psi_j) \rangle = 0.
\]

\vspace{0.2cm}

\noindent $\bullet$ (Estimate of ${\mathcal I}_{12}$ and ${\mathcal I}_{13}$):  Since $V_k$ and $V_j$ are real-valued functions, we find that
\begin{eqnarray*}
{\mathcal I}_{12}\!+\!{\mathcal I}_{13} &=& -\! \frac{{\mathrm i}}{\hbar} \langle V_k\psi_k | \psi_k-\psi_j \rangle\!+\! \frac{{\mathrm i}}{\hbar} \langle V_j\psi_j | \psi_k-\psi_j \rangle 
								+ \frac{{\mathrm i}}{\hbar} \langle \psi_k - \psi_j | V_k\psi_k \rangle - \frac{{\mathrm i}}{\hbar} \langle \psi_k - \psi_j | V_j\psi_j \rangle 
\cr
&=&\frac{{\mathrm i}}{\hbar} \int_{\bT^d} (V_k - V_j) (\psi_k {\overline \psi}_j - {\overline \psi}_k \psi_j ) dx = -\frac{2}{\hbar} \int_{\bT^d} (V_k - V_j) \mbox{Im}(\psi_k {\overline \psi}_j) dx.
\end{eqnarray*}      

\vspace{0.2cm}

\noindent $\bullet$ (Estimate of ${\mathcal I}_{14} + {\mathcal I}_{15}$): Similar calculations show that
\[
\begin{aligned}
{\mathcal I}_{14} + {\mathcal I}_{15} =& \frac{{\mathrm i}}{\hbar} \int_{\bT^d} \left(\omega \star_x (|\psi_k|^2 - |\psi_j|^2)\right) (\psi_k {\overline \psi}_j - {\overline \psi}_k \psi_j) dx 
\\
=& -\frac{2}{\hbar} \int_{\bT^d} \left(\omega \star_x (|\psi_k|^2 - |\psi_j|^2)\right) \mbox{Im}(\psi_k {\overline \psi}_j) dx.
\end{aligned}
\]

\vspace{0.2cm}

\noindent $\bullet$ (Estimate of ${\mathcal I}_{16} + {\mathcal I}_{17}$): Proceeding as in \cite{C-C-H}, we find that
\begin{eqnarray*}
{\mathcal I}_{16} + {\mathcal I}_{17} &=& \frac{\kappa}{2N}\sum_{m=1}^N\Big( \langle \psi_j | \psi_m \rangle\langle \psi_j | \psi_k-\psi_j \rangle -\langle\psi_k | \psi_m\rangle \langle \psi_k | \psi_k-\psi_j \rangle  
\cr
&+& \langle \psi_m | \psi_j \rangle\langle\psi_k - \psi_j | \psi_j \rangle -\langle\psi_m | \psi_k \rangle \langle  \psi_k - \psi_j | \psi_k \rangle \Big) 
\cr
&=&  \frac{\kappa}{N}\sum_{m=1}^N {\mbox Re}\bigg[ \langle \psi_j | \psi_m \rangle\langle \psi_j | \psi_k-\psi_j \rangle -\langle\psi_k | \psi_m\rangle \langle \psi_k| \psi_k-\psi_j \rangle\bigg] 
\cr
&\leq&  \frac{\kappa}{N}\sum_{m=1}^N \Big[  \Big(||\psi_k-\psi_m||_{\fH} +||\psi_j-\psi_m||_{\fH}  - 1\Big)\cdot||\psi_k-\psi_j||_{\fH}^2-|1-\langle \psi_k | \psi_j \rangle|^2 \Big].
\end{eqnarray*}
Finally, in \eqref{B-8}, we combine all the inequalities for ${\mathcal I}_{1i}$ to obtain the desired estimate.
\end{proof}

\subsection{Reformulation of the SL model} \label{sec:2.3}

Let $U(t)$ be the free Schr\"odinger flow with $\hbar = 1$ denoted by
\[ 
U(t):=e^{{\mathrm i}t\Dlt/2}\,,\quad t\in \bbr,  
\]
and we set 
\[ 
\phi_k(t):=U(\hbar t)^*\psi_k(t)\,,\quad k \in [N]. 
\]
Then, $\phi_k$ satisfies
\[ 
{\mathrm i} \hbar\d_t\phi_k=U(\hbar t)^*(({\mathrm i}\hbar\d_t+\tfrac12\hbar^2\Dlt_x)\psi_k) \quad \mbox{for each $k \in [N]$}. 
\]
Hence the Cauchy problem \eqref{SL} becomes
\begin{equation}\label{B-14}
\begin{cases}
\displaystyle \d_t\phi_k \!=\!\tfrac1{i\hbar}U(\hbar t)^*(V_k\!+\!\om\star_x|U(\hbar t)\phi_k|^2)U(\hbar t)\phi_k \!+\!\frac{\kappa}{2N}\sum_{l=1}^N \Big(\phi_l\!-\!\la\phi_k|\phi_l\ra\phi_k \Big ),~~x\in\bbt^d,
\\
\displaystyle \phi_k\rstr_{t=0}=\psi_k^{0}\,,\quad k \in [N]\,.
\end{cases}
\end{equation}
Since $U(t)$ is a unitary group on $L^2(\bbt^d)$, one has
$$
\la\phi_k(t)|\phi_l(t)\ra=\la\psi_k(t)|\psi_l(t)\ra\,,\quad k,l \in [N]\,.
$$
We set 
\[
X_k(t,\phi_1,\ldots,\phi_N):=\frac1{i\hbar}U(\hbar t)^*(V_k+\om\star_x|U(\hbar t)\phi_k|^2)U(\hbar t)\phi_k +\frac{\kappa}{2N}\sum_{l=1}^N \Big(\phi_l-\la\phi_k|\phi_l\ra\phi_k \Big),
\]
for $k \in [N]$ and $X(t,\phi_1,\ldots,\phi_N):=(X_k(t,\phi_1,\ldots,\phi_N))_{1\le k\le N}$. Clearly $X_k$ is continuous on $\bbr\times\fH^N$ for each $k=1,\ldots,N$, and Lipschitz continuous in 
$\phi_1,\ldots,\phi_N\in B(0,R)$ for each $R>0$ uniformly in $t\in\bbr$. By the Cauchy-Lipschitz theorem recalled in Appendix \ref{S-CL}, the differential system \eqref{B-14} has a unique local smooth solution.
Next we observe that
\begin{align*}
\begin{aligned}
& X_k(t,\phi_1,\ldots,\phi_N)|\phi_k\ra \\
& \hspace{0.5cm} =\La\frac1{i\hbar}U(\hbar t)^*(V_k+\om\star_x|U(\hbar t)\phi_k|^2)U(\hbar t)\phi_k\Big|\phi_k\Ra +\frac\kappa2(\la\phi_c|\phi_k\ra-\la\phi_k|\phi_c\ra\|\phi_k\|_\fH^2)
\\
& \hspace{0.5cm} =\La\frac1{i\hbar}(V_k+\om\star_x|U(\hbar t)\phi_k|^2)U(\hbar t)\phi_k\Big|U(\hbar t)\phi_k\Ra
\\
& \hspace{0.7cm} +\frac\kappa2(\la U(\hbar t)\phi_c|U(\hbar t)\phi_k\ra-\la U(\hbar t)\phi_k|U(\hbar t)\phi_c\ra\|\phi_k\|_\fH^2)
\\
& \hspace{0.5cm} =\La\frac1{i\hbar}(V_k+\om\star_x|\psi_k|^2)\psi_k\Big|\psi_k\Ra+\frac\kappa2(\la\psi_c|\psi_k\ra-\la\psi_k|\psi_c\ra\|\psi_k\|_\fH^2),
\end{aligned}
\end{align*}
with $\phi_c:=\frac1N\sum_{l=1}^N\phi_l$  so that the same computation as in the proof of Lemma \ref{L2.1} shows that $\|\phi_k(t)\|_\fH=1$ for each $k \in [N]$ and each $t$ in the maximal interval of definition of the solution of \eqref{B-14}.
Hence, one has 
\[
\|\partial_t\phi_k(t)\|_\fH\le\tfrac1{\hbar}(\|V_k\|_{L^\infty}+\|\omega\|_{L^\infty})+\kappa\,,
\]
so that
\[
\|\phi_k(t_2)-\phi_k(t_1)\|_\fH\le\int_{t_1}^{t_2}\|\partial_t\phi_k(t)\|_\fH dt\le\int_0^T\|\partial_t\phi_k(t)\|_\fH dt<\infty,
\]
whenever $0<t_1<t_2<T<\infty$. By the Cauchy convergence criterion, $\phi_k(t)$ has a limit in the complete space $\fH$ as $t\to T^-$ for each $T\in(0,+\infty)$. As recalled in Appendix \ref{S-CL},
this implies that the differential system \eqref{B-14} has a unique global solution on $\bS\fH^N$ for all $\psi_1^{0},\ldots,\phi_N^{0}\in\bS\fH$.

Since $U(t)$ is a unitary group on $\fH$, one easily checks that the conservation of $L^2$-norm in Lemma \ref{L2.1} applies to $\psi_k(t)=U(-\hbar t)\phi_k(t)$ for $k=1,\ldots,N$ for all $t\in\bbr$, for 
all initial data $\psi_k^0\in\bS\fH$, without having to postulate any smoothness condition on either $\psi_k^0$ or $\psi_k(t)$.

Similarly, the conclusion of Lemma \ref{L2.2} also holds for all $t\in[0,+\infty)$ in its time-integrated form, without having to postulate any smoothness condition on either $\psi_k^0$ or $\psi_k(t)$. 
To see this, we first assume that the potentials $V_k$, the Hartree kernel $\omega$ and the initial data $\psi_k^0$ are smooth on $\bbt^d$. Then, the global solution $\{\phi_k(t)\}$ for $k \in [N]$
remains smooth for all $t\ge 0$, and the conclusion of Lemma \ref{L2.2} holds for all $t\in[0,+\infty)$. After we integrate \eqref{B-8-1} in $t$ to find 
\begin{align*} \label{Int<L2.2}
\begin{aligned} 
&\|\psi_k(t_2)-\psi_j(t_2)\|_{\fH}^2 \\
& \hspace{0.5cm} \leq \|\psi_k(t_1)-\psi_j(t_1)\|_{\fH}^2 \!-\!\frac{2}{\hbar}\int_{t_1}^{t_2}\int_{\bbt^d}\Big((V_k\!-\!V_j)\!+\!\omega\star_x (|\psi_k|^2\!-\!|\psi_j|^2) \Big )\mbox{Im}(\psi_k{\overline\psi}_j)dxdt
\\
& \hspace{0.7cm} +\frac{\kappa}{N}\sum_{m=1}^N\int_{t_1}^{t_2} \Big (\left(\|\psi_k\!-\!\psi_m\|_{\fH}\!+\!\|\psi_j\!-\!\psi_m\|_{\fH}\!-\!1\right)\|\psi_k\!-\!\psi_j\|_{\fH}^2\!-\!|1\!-\!\langle \psi_k|\psi_j \rangle|^2 \Big )dt,
\end{aligned}
\end{align*}
for all $0\le t_1<t_2<\infty$. The case of general potentials $V_k\in L^\infty(\bbt^d)$ and of $\omega\in L^\infty(\bbt^d)$ follows by the density argument in the $L^2$ topology induced on $\bS\fH$. 

In the sequel, we conclude Section \ref{sec:2.3} with a remark. Consider the Cauchy problem:
\begin{equation} \label{B-15}
\begin{cases}
\displaystyle \d_t\phi_k = \frac{\kappa}{2N}\sum_{l=1}^N \Big(\phi_l-\la\phi_k|\phi_l\ra\phi_k \Big ), \quad (t, x) \in \bbr_+ \times \bbt^d,
\\
\displaystyle \phi_k\rstr_{t=0}=\psi_k^{0}, \quad k \in [N],
\end{cases}
\end{equation}
and denote its solution operator for \eqref{B-15} by $L_j (t)\phi^{0} = \phi_j(t,\cdot)$. In the case of equal one-body potentials $V_k = V$, the SL model enjoys a solution splitting property which says that 
the solution to \eqref{SL} can be expressed as the composition of the free Schr\"odinger flow and the nonlinear Lohe flow.

\begin{Prop}\label{P2.1}
\emph{\cite{H-H-K1}}
Suppose that the potentials and the Hartree convolution kernel satisfy
\[ 
V_k = V, \quad  k \in [N] \quad \mbox{and} \quad \omega = 0,
\]
and let $\Psi := \{ \psi_k \}$ be a global smooth solution to \eqref{SL}. Then, the following solution splitting property holds:
\[ 
\psi_j(t,x) = \exp\left(\tfrac{it}2\Delta-\tfrac{i}{\hbar^2}V\right)\circ L_j(t) \psi^{0}_{j}(x), \quad (t, x) \in \bbr_+ \times \bbt^d. 
\]
\end{Prop}

\begin{proof} 
For a proof, we refer to Proposition 2.3 of \cite{H-H-K1}.
\end{proof}

\subsection{Practical synchronization} \label{sec:2.4}
Now, we are ready to discuss the  practical synchronization of \eqref{SL}. For this, we set
\[
 {\mathcal D}({\mathcal V}) := \max_{1\le k, l\le N} \| V_k - V_l \|_{L^{\infty}}, \quad {\mathcal D}(\Psi) := \max_{1\le k, l\le N} \|\psi_k - \psi_l \|_{\fH}.
\]
Consider the following cubic polynomial: 
\[
q(x) := 2 x^3- x^2+ \frac{1}{\kappa \hbar} \Big( 2 {\mathcal D}({\mathcal V}) +   4 ||\omega||_{L^{\infty}} \Big), \quad x \in [0, \infty). 
\]
We assume that the coupling strength $\kappa$ is sufficiently large such that 
\[
  \kappa > \frac{27}{\hbar} \Big(  2 {\mathcal D}({\mathcal V}) +  4 ||\omega||_{L^{\infty}}  \Big). 
\]
Then, $q$ has a negative global minimum $\displaystyle -\frac{1}{27} + \frac{1}{\kappa \hbar} \Big( 2{\mathcal D}({\mathcal V}) +   4 ||\omega||_{L^{\infty}} \Big)$ at $\displaystyle  x = \frac{1}{3}$. Moreover, $q(x)= 0$ has two positive real roots $\alpha_1 <  \alpha_2 $:
\[  
0 < \alpha_1 < \frac{1}{3} < \alpha_2 < \frac{1}{2}. 
\]
In addition, since the roots $\alpha_1$ and $\alpha_2$ depend continuously on $\kappa$ and ${\mathcal D}({\mathcal V})$:
\begin{equation} \label{B-9}
 \lim_{\kappa \to \infty} \alpha_1 = 0,  \qquad  \lim_{\kappa \to \infty} \alpha_2 = \frac{1}{2}.
\end{equation}

\begin{Thm}\lb{T-PractiSynchroSL}
\emph{(Practical synchronization)} Suppose that the one-body potentials $V_k$, the convolution kernel $\omega$, the coupling strength $\kappa$ and the initial data satisfy
\[ 
\max_{1 \leq k \leq N} \| V_k  \|_{L^{\infty}} + \| \omega \|_{L^\infty} < \infty, \qquad \kappa > \frac{27}{\hbar} \Big(  2 {\mathcal D}({\mathcal V}) +  4 ||\omega||_{L^{\infty}}  \Big), \qquad {\mathcal D}(\Psi^{{0}}) < \alpha_2,
\]
and let $ \Psi = \{ \psi_k \}$ be a global smooth solution to \eqref{SL}. Then, we have 
\[ 
\lim_{\kappa \to \infty} \varlimsup_{t \to \infty} {\mathcal D}(\Psi(t)) = 0. \]
\end{Thm}
\begin{proof} We split the proof into two steps. For notational simplicity, we suppress $t$-dependence in $\psi_i$:
\[ \psi_i(x) \equiv \psi_i(t,x), \quad i \in [N]. \]

\vspace{0.1cm}

\noindent $\bullet$ Step A (Derivation of a differential inequality for ${\mathcal D}(\Psi)$): First, we note that 
\begin{eqnarray*}
&& |\mbox{Im}(\psi_k {\overline \psi}_j) | \leq |\psi_k {\overline \psi}_j| = |\psi_k|  |{\overline \psi}_j| \leq \frac{1}{2} \Big( |\psi_k|^2 + |\psi_j|^2 \Big), 
\cr
&& |\omega * (|\psi_k|^2 - |\psi_j|^2)(x)| = \Big| \int_{\bbt^d} \omega(x-y)  (|\psi_k(y)|^2 - |\psi_j(y)|^2) dy  \Big| 
\cr
&& \hspace{2cm} \leq ||\omega||_{L^{\infty}} \int_{\bbt^d} \Big (|\psi_k(y)|^2 + |\psi_j(y)|^2 \Big ) dy = 2||\omega||_{L^{\infty}}. 
\end{eqnarray*} 
Thus, we have
\begin{align*}
\begin{aligned}
\frac{d}{dt}||\psi_k-\psi_j||_{\fH}^2 &\leq \frac{2}{\hbar} ||V_k - V_j||_{L^{\infty}} +  \frac{4 ||\omega||_{L^{\infty}}}{\hbar} -\ka|1-\langle\psi_k|\psi_j\rangle|^2
\\
&+ \frac{\kappa}{N}\sum_{m=1}^N \Big[  \Big(||\psi_k-\psi_m||_{\fH} +||\psi_j-\psi_m||_{\fH}  - 1\Big)\cdot||\psi_k-\psi_j||_{\fH}^2 \Big].
\end{aligned}
\end{align*}
This yields 
\begin{align}
\begin{aligned} \label{B-10}
\frac{d}{dt} {\mathcal D}(\Psi)^2 &\leq \frac{2{\mathcal D}({\mathcal V})}{\hbar} +   \frac{4 ||\omega||_{L^{\infty}}}{\hbar}  + \kappa (2 {\mathcal D}(\Psi) - 1) {\mathcal D}(\Psi)^2 
\\
&= \kappa \Big[ 2 {\mathcal D}(\Psi)^3 - {\mathcal D}(\Psi)^2 +  \frac{2{\mathcal D}({\mathcal V})}{\hbar \kappa} +   \frac{4 ||\omega||_{L^{\infty}}}{\hbar \kappa}  \Big].
\end{aligned}
\end{align}
Notice here again that we have chosen to use the ``formal'' differential form instead of the rigorous time-integrated form of \eqref{B-10}, for the sake of clarity.

\vspace{0.2cm}

\noindent $\bullet$ Step B (Upper bound estimate for ${\mathcal D}(\Psi)$):~For a relative size of ${\mathcal D}(\Psi^{0})$, we consider two cases. \newline

\noindent $\diamond$ Case A: Suppose that $\Psi^{0}$ satisfies
\[ 
{\mathcal D}(\Psi^{0}) \leq \alpha_1. 
\]
Let $t = t_*$ be the instant such that 
\[ 
{\mathcal D}(\Psi(t_*)) = \alpha_1. 
\]
Then, it follows from \eqref{B-10} that 
\[  
\frac{d}{dt} \Big|_{t = t_*} {\mathcal D}(\Psi)^2 \leq 0.
\]
Thus, ${\mathcal D}(\Psi)$ is confined in the interval $[0, \alpha_1]$ for all $t \geq 0$:
\begin{equation} \label{B-11}
0 \leq  {\mathcal D}(\Psi(t)) \leq \alpha_1.
\end{equation}

\noindent $\diamond$ Case B: Suppose that $\Psi^{0}$ satisfies
\[ 
\alpha_1 < {\mathcal D}(\Psi^{0}) < \alpha_2. 
\]
In this case, it is easy to see that at the instant $t$ such that  ${\mathcal D}(\Psi(t))$ is restricted within the interval $(\alpha_1, \alpha_2)$, we have
\[  
\frac{d}{dt}({\mathcal D}(\Psi)(t))^2 \leq \kappa \Big[ ({\mathcal D}(\Psi)(t))^2(2{\mathcal D}(\Psi)(t)-1)+  \frac{2{\mathcal D}({\mathcal V})}{\hbar \kappa} +   \frac{4 ||\omega||_{L^{\infty}}}{\hbar \kappa} \Big] < 0. 
\]
In other words, ${\mathcal D}(\Psi)$ is decreasing in the initial phase of its evolution. Then, by the same argument in Proposition 3.1  \cite{C-H-J-K}, there exists a positive time $t_e \in (0, \infty)$ such that
\begin{equation} \label{B-12}
{\mathcal D}(\Psi(t)) \leq \alpha_1, \quad t \geq t_e. 
\end{equation}
We combine \eqref{B-11} and \eqref{B-12} to find
\begin{equation} \label{B-13}
 0 \leq \limsup_{t \to \infty} {\mathcal D}(\Psi(t)) \leq \alpha_1. 
\end{equation} 
Finally, we combine \eqref{B-9} and \eqref{B-13} to obtain the desired practical synchronization:
\[ 
\lim_{\kappa \to \infty} \limsup_{t \to \infty} {\mathcal D}(\Psi(t)) = 0. 
\]
\end{proof}


\section{The Liouville equation and the kinetic SL equation} \label{sec:3}


In this section, we consider the Liouville equation for the $N$-oscillator distribution function (probability measure) on the $N$-oscillator phase space, and then derive the mean-field kinetic equation --- which can be obtained via 
the molecular chaos assumption. For the related mean-field limit of the Schr\"odinger equation and Lohe matrix model, we refer to the recent works \cite{G-H, G-P, G-M-P}.

\subsection{The Liouville equation}  \label{sec:3.1}


First, we define the $N$-oscillator phase space as 
\[
(\bS\fH)^N\times C(\bbt^d)^N
\]
(where we recall that $\bS\fH$ is the unit sphere of the Hilbert space $\fH=L^2(\bbt^d)$). A point in this $N$-oscillator phase-space is $(\psi_1,\ldots,\psi_N,V_1,\ldots,V_N)$, where $\psi_j$ is the $j$-th oscillator's wave 
function, while $V_j$ is the applied potential in the equation for $\d_t\psi_j$. (Although including $V_1,\ldots,V_N$ in the phase space might seem useless at this point, there is a very good reason for doing so, which 
we shall explain later). The $N$-oscillator distribution function $F_N$ is a time-dependent Borel probability measure on the $N$-oscillator phase space, denoted by
\[
F_N(t)\equiv F_N(t,d\psi_1\ldots d\psi_N,dV_1\ldots dV_N)\,, 
\]
where
\[
\bbr\ni t\mapsto F_N(t)\in\cP((\bS\fH)^N\times C(\bbt^d)^N)\,.
\]
For each separable complete metric (i.e. Polish) space  $Z$, we denote by $\cP(Z)$ the set of all Borel probability measures on $Z$. 

Recall that the evolution of the transformed wave function $\phi_k(t):=U(-\hbar t)\psi_k(t)$ is governed by the following Cauchy problem (see \eqref{B-14}):
\begin{equation} \label{C-1}
\begin{cases}
\displaystyle \d_t\phi_k\!=\!\frac1{i\hbar}U(\hbar t)^*(V_k\!+\!\om\star_x|U(\hbar t)\phi_k|^2)U(\hbar t)\phi_k\!+\!\frac{\kappa}{2N}\sum_{l=1}^N(\phi_l\!-\!\la\displaystyle \phi_k|\phi_l\ra\phi_k),~~x\in\bbt^d, 
\\
\phi_k\rstr_{t=t_0}=\psi_k^{0}\, \qquad k \in [N].
\end{cases}
\end{equation}
In the sequel, we view the SL model \eqref{C-1} in the $\phi_k$ coordinates as a system of ODEs written on the phase space $(\bS\fH)^N\times C(\bbt^d)^N$, as explained in 
Section \ref{sec:2.3}, and one defines
\begin{equation}\label{Def-SN(tt0)}
\cS_N(t,t_0)(\phi_1^{0},\ldots,\phi_N^{0},V_1,\ldots,V_N):=(\phi_1(t),\ldots,\phi_N(t),V_1,\ldots,V_N)\,,\quad t\in\bbr\,.
\end{equation}
The global existence of the flow $\cS_N(t,t_0)$ for all $t,t_0\in\bbr$ follows from the application of the Cauchy-Lipschitz Theorem to \eqref{C-1} as explained in Section \ref{sec:2.3}. Moreover, Lemma \ref{L2.1} implies that 
$\cS_N(t,t_0)\left((\bS\fH)^N\times C(\bbt^d)^N\right)\subset(\bS\fH)^N\times C(\bbt^d)^N$, as explained in Section \ref{sec:2.3}. 

Thus, the right-hand side of \eqref{C-1} can be regarded as a vector field on $(\bS\fH)^N\times C(\bbt^d)^N$, which is an (infinite dimensional) smooth manifold --- in fact a submanifold of $\fH^N\times C(\bbt^d)^N$.
More precisely, we decompose this right-hand side as follows:
\be\label{C-1-1}
\left\{
\begin{aligned} 
{}& X_N =X^f_N+X^i_N, 
\\
& X^f_N(t,\phi_1,\ldots,\phi_N,V_1,\ldots,V_N):=\Big(\tfrac1{{\mathrm i}\hbar}U(\hbar t)^*(V_k+\om\star_x|U(\hbar t)\phi_k|^2)U(\hbar t)\phi_k \Big)_{1\le k\le N}\oplus 0_N,
\\
& X^i_N(t,\phi_1,\ldots,\phi_N,V_1,\ldots,V_N) := \left(\frac{\kappa}{2N}\sum_{l=1}^N(\phi_l-\la\phi_k|\phi_l\ra\phi_k)\right)_{1\le k\le N}\oplus 0_N\,.
\end{aligned}
\right.
\ee
The tangent space of $(\bS\fH)^N\times C(\bbt^d)^N$ at $(\phi_1,\ldots,\phi_N,V_1,\ldots,V_N)$ is
\[
\begin{aligned}
T_{(\phi_1,\ldots,\phi_N)}(\bS\fH)^N\times C(\bbt^d)^N &=~(T_{(\phi_1,\ldots,\phi_N)}(\bS\fH)^N\times\{0_N\})\oplus(\{0_N\}\times C(\bbt^d)^N)
\\
& \simeq T_{(\phi_1,\ldots,\phi_N)}(\bS\fH)^N\oplus C(\bbt^d)^N\,.
\end{aligned}
\]
as explained in Appendix \ref{SLiouv}. If $Y\in T_{(\phi_1,\ldots,\phi_N)}(\bS\fH)^N$, the notation $Y\oplus 0_N$ designates 
\[
(Y,0_N)\in T_{(\phi_1,\ldots,\phi_N)}(\bS\fH)^N\times\{0_N\}\subset T_{(\phi_1,\ldots,\phi_N)}(\bS\fH)^N\times C(\bbt^d)^N\,.
\]
Notice that $X_N^f$ and $X_N^i$ correspond to the Schr\"odinger flow part and the Lohe coupling part, respectively. Next, we observe that
\[
\La\tfrac1{{\mathrm i} \hbar}(V_k+\om\star_x|U(\hbar t)\phi_k|^2|)\phi_k \Big |\phi_k\Ra_\fH=0\,,\quad k \in [N], 
\]
since the multiplication by the bounded real-valued function
\[
V_k+\om\star_x|U(\hbar t)\phi_k|^2
\]
is a self-adjoint operator on $\fH$. Hence, according to \eqref{TSH}
\[
X^f_N(t,\phi_1,\ldots,\phi_N,V_1,\ldots,V_N)_k\in T_{\phi_k}\bS\fH\,,\quad k \in [N],
\]
so that, by \eqref{TangSpace},
\[
X^f_N(t,\phi_1,\ldots,\phi_N,V_1,\ldots,V_N)\in T_{\phi_1}\bS\fH\times\ldots\times T_{\phi_N}\bS\fH=T_{(\phi_1,\ldots,\phi_N)}(\bS\fH)^N\,.
\]
Similarly, we have
\[
\RE\La\sum_{l=1}^N(\phi_l-\la\phi_k|\phi_l\ra\phi_k)\Big|\phi_k\Ra=\RE(\la\phi_c|\phi_k\ra-\la\phi_k|\phi_c\ra)=0\, \quad \text{with }\phi_c=\frac1N\sum_{l=1}^N\phi_l\,,
\]
so that, according to \eqref{TSH},
\[
X^i_N(t,\phi_1,\ldots,\phi_N,V_1,\ldots,V_N)_k\in T_{\phi_k}\bS\fH\,,\qquad k \in [N]\,,
\]
and hence, by \eqref{TangSpace},
\[
X^i_N(t,\phi_1,\ldots,\phi_N,V_1,\ldots,V_N)\in T_{\phi_1}\bS\fH\times\ldots\times T_{\phi_N}\bS\fH=T_{(\phi_1,\ldots,\phi_N)}(\bS\fH)^N\,.
\]

Choosing some initial $N$-oscillator distribution function $F_N^{0}\in\cP((\bS\fH)^N\times C(\bbt^d)^N)$ at time $t=0$, we define the $N$-oscillator distribution function $F_N(t)$ at time $t$ by the 
prescription\footnote{Let $E$ and $F$ be two sets, let $\mathcal A$ and $\mathcal B$ be $\sigma$-algebras on $E$ and $F$ respectively, and let $T:\,E\to F$ be an $\mathcal A-\mathcal B$-measurable map. 
We set $T\#\mu(B):=\mu(T^{-1}B)$ for all $B\in\mathcal B$, where $\mu$ is a probability measure on the measurable space $(E,\mathcal A)$. Then $T\#\mu$ is a probability measure on $(F,\mathcal B)$ 
(the push-forward of $\mu$ by $T$).} 
\[
\lb{DefFN(t)}
F_N(t):=\cS_N(t,0)\#F_N^{0}\,.
\]
Therefore, $F_N(t)$ so defined is the ``{\it Lagrangian weak solution}" of the Liouville equation on the $N$-oscillator phase space according to Lemma \ref{L-B.1} --- see also Definition \ref{D-B.1} for the notion
of divergence of a vector field on $(\bS\fH)^N\times C(\bbt^d)^N$ in Appendix \ref{SLiouv}:
\begin{equation}
\begin{cases}  \label{C-2}
\displaystyle  \d_tF_N+\Div(F_NX_N)=0, \quad (t, \phi, V) \in \bbr \times (\bS\fH)^N\times C(\bbt^d)^N,
\\
\displaystyle F_N\rstr_{t=0}=F_N^{0}\in\cP((\bS\fH)^N\times C(\bbt^d)^N)\,.
\end{cases}
\end{equation}

Note that since the vector field $X$ is of the form $X=Y\oplus 0_N$, one could consider $(\bS\fH)^N$ as the phase space and drop the dependence of $F_N$ on $V_k$. However, doing so would destroy 
the propagation of symmetry by $\cS_N(t,0)$, whose importance for the mean-field limit will become clear in the next subsection.

\subsection{Derivation of the kinetic SL equation} \label{sec:3.2}

In this subsection, we derive a mean-field equation for an approximate one-oscillator distribution function via the BBGKY hierarchy argument.

\subsubsection{Symmetric $N$-oscillator distributions}

For each $\si\in\fS_N$, we define
\begin{equation*}\label{DefTsigma}
T_\si(\phi_1,\ldots,\phi_N,V_1,\ldots,V_N):=(\phi_{\si(1)},\ldots,\phi_{\si(N)},V_{\si(1)},\ldots,V_{\si(N)}).
\end{equation*}
Suppose that the initial $N$-particle distribution function is symmetric in each argument, in other words
\begin{equation*}\label{SymFN0}
T_\si\# F^0_N=F^0_N\,,\quad\hbox{ for all }\si\in\fS_N\,.
\end{equation*}
Then, this symmetry property of the $N$-particle distribution function is propagated by the dynamics of the $N$-particle Liouville equation \eqref{C-2}:
\[
T_\si\#F_N(t)=F_N(t)\,,\quad\hbox{ for all }\si\in\mathfrak S_N\hbox{ and all }t\in \bbr.
\]
To see this, we write the vector field $X^f_N$ in \eqref{C-1-1} in the following form:
\begin{eqnarray*}
&& X^f_N(t,\phi_1,\ldots,\phi_N,V_1,\ldots,V_N)=(X^f_N(t,\phi_1,\ldots,\phi_N,V_1,\ldots,V_N)_k)_{1\le k\le N}\oplus 0_N\,, 
\\
&&  X^f_N(t,\phi_1,\ldots,\phi_N,V_1,\ldots,V_N)_k:=\tfrac1{{\mathrm i} \hbar}U(\hbar t)^*(V_k+\om\star_x|U(\hbar t)\phi_k|^2)U(\hbar t)\phi_k.
\end{eqnarray*}
Then, it is easy to see
\[
X^f_N(t,T_\si(\phi_1,\ldots,\phi_N,V_1,\ldots,V_N))=(X^f_N(t,\phi_1,\ldots,\phi_N,V_1,\ldots,V_N)_{\si(k)})_{1\le k\le N}\oplus 0_N
\]
for each $\si\in\fS_N$, since $X^f_N(t,\phi_1,\ldots,\phi_N,V_1,\ldots,V_N)_k$ depends on $\phi_k,V_k$ only.  Likewise, we write the vector field $X^i_N$ in the following form:
\[
X^i_N(t,\phi_1,\ldots,\phi_N,V_1,\ldots,V_N)=(X^i_N(t,\phi_1,\ldots,\phi_N,V_1,\ldots,V_N)_k)_{1\le k\le N}\oplus 0_N\,.
\]
Then, we can see 
\[
X^i_N(t,T_\si(\phi_1,\ldots,\phi_N,V_1,\ldots,V_N))=(X^i_N(t,\phi_1,\ldots,\phi_N,V_1,\ldots,V_N)_{\si(k)})_{1\le k\le N}\oplus 0_N\,,
\]
since
\begin{align*}
\begin{aligned}
& X^i_N(t,T_\si(\phi_1,\ldots,\phi_N,V_1,\ldots,V_N))_k 
\\
& \hspace{0.5cm} =\frac{\kappa}{2N}\sum_{l=1}^N \Big(\phi_{\si(l)}-\la\phi_{\si(k)}|\phi_{\si(l)}\ra\phi_{\si(k)} \Big)=\frac{\kappa}{2N}\sum_{j=1}^N \Big(\phi_j-\la\phi_{\si(k)}|\phi_{j}\ra\phi_{\si(k)} \Big) 
\\
& \hspace{0.5cm} =X^i_N(t,\phi_1,\ldots,\phi_N,V_1,\ldots,V_N)_{\si(k)}\,,
\end{aligned}
\end{align*}
where we used the substitution $j=\si(l)$ in the sum. These identities satisfied by the vector fields $X^f_N$ and $X^i_N$ imply that
\begin{equation} \label{New-1}
T_\si\circ\cS_N(t,t_0)=\cS_N(t,t_0)\circ T_\si\,,\quad t,t_0\in\bbr,~~\si\in\fS_N\,,
\end{equation}
by the uniqueness of the solution of the differential system \eqref{C-1} implied by the Cauchy-Lipschitz theorem recalled in Appendix \ref{S-CL}. Hence we have
\begin{align*}
\begin{aligned} \lb{SymFN(t)}
T_\si\#F_N(t) &=T_\si\#(\cS_N(t,0)\#F_N^0)=T_\si\circ\cS_N(t,0)\#F_N^0=\cS_N(t,0)\circ T_\si\#F_N^0
\\
&=\cS_N(t,0)\#(T_\si\#F_N^0)=\cS_N(t,0)\#F_N^0=F_N(t),
\end{aligned}
\end{align*}
for all $t\in\bbr$ and all $\si\in\fS_N$.

\subsubsection{Marginals of symmetric $N$-oscillator distributions}

For each $m \in [N]$, the $m$-particle marginal of a symmetric $N$-particle distribution $F_N$ is
\[
F_{N:m}\in\cP((\bS\fH)^m\times C(\bbt^d)^m)\,,
\]
which is defined by the relation:
\begin{align*}
\begin{aligned}
& \int_{(\bS\fH)^m\times C(\bT^d)^m}\chi_m(\phi_1,\ldots,\phi_m,V_1,\ldots,V_m)F_{N:m}(d\phi_1,\ldots,d\phi_m,dV_1,\ldots,dV_m)
\\
& \hspace{0.7cm} :=\int_{(\bS\fH)^N\times C(\bT^d)^N}\chi_m(\phi_1,\ldots,\phi_m,V_1,\ldots,V_m)F_N(d\phi_1,\ldots,d\phi_N,dV_1,\ldots,dV_N)
\end{aligned}
\end{align*}
for each test function $\chi_m\in C_b((\bS\fH)^m\times C(\bbt^d)^m)$. In particular, we apply this identity for $\chi_m\equiv 1$ to see that
\begin{align*}
\begin{aligned}
& \int_{(\bS\fH)^m\times C(\bT^d)^m}F_{N:m}(d\phi_1,\ldots,d\phi_m,dV_1,\ldots,dV_m) 
\\
& \hspace{1cm} =\int_{(\bS\fH)^N\times C(\bT^d)^N}F_N(d\phi_1,\ldots,d\phi_N,dV_1,\ldots,dV_N)=1.
\end{aligned}
\end{align*}
In other words,
\[  
F_{N:m}:=\Pi_m\#F_N\,, 
\]
where $\Pi_m$ is the projection on the $m$-particle phase space:
\[
\ba
\Pi_m:\,(\bS\fH)^N\times C(\bbt^d)^N&\to(\bS\fH)^m\times C(\bbt^d)^m
\\
(\phi_1,\ldots,\phi_N,V_1,\ldots,V_N)&\mapsto(\phi_1,\ldots,\phi_m,V_1,\ldots,V_m)\,.
\ea
\]

\subsubsection{Derivation of the kinetic equation}

Recall that 
\begin{equation}\label{NotPhiNVN}
\Phi_N:=(\phi_1,\ldots,\phi_N) \quad\text{ and }\quad\cV_N:=(V_1,\ldots,V_N). 
\end{equation}
At this point, we introduce the following vector fields:
\be\lb{DefY}
\begin{aligned}
&Y^f(t,\phi,V):=\tfrac1{i\hbar}U(\hbar t)^*(V+\om\star_x|U(\hbar t)\phi|^2)U(\hbar t)\phi\,, 
\\
&Y^i_k(t,\Phi_N):=\frac{1}{N}\sum_{l=1}^N(\phi_l-\la\phi_k|\phi_l\ra\phi_k) =(I-\bP_{\phi_k})\frac{1}{N}\sum_{l=1}^N\phi_l. 
\end{aligned}
\ee
Here we denote by $\bP_\phi$ the $L^2(\bbt^d)$-orthogonal projection on the complex line $\bbc \phi$:
\[ 
\bP_\phi\tilde\phi=\la\phi|\tilde\phi\ra\phi\,,\qquad\|\phi\|_\fH=1\,. 
\]
With the notation introduced so far, we can write the equation for the first marginal of $F_N$:
\begin{align*}
\begin{aligned} 
& \d_tF_{N:1}+\Div\left(F_{N:1}\left(\tfrac1{i\hbar}U(\hbar t)^*(V_1+\om\star_x|U(\hbar t)\phi_1|^2)U(\hbar t)\phi_1\oplus 0\right)\right) 
\\
& \hspace{3cm} + \frac{\kappa (N-1)}{2N}\Div\left(\left(F_{N:2}(\phi_2-\la\phi_1|\phi_2\ra\phi_1)\oplus 0_2\right)_{:1}\right)=0.
\end{aligned}
\end{align*}
Indeed, using the weak formulation of \eqref{C-2}, for each $\chi\in C^1_b(\bS\fH\times C(\bbt^d))$, we have
\begin{align}
\begin{aligned} \label{C-3}
& \frac{d}{dt}\int_{\bS\fH\times C(\bbt^d)} \chi(\phi_1,V_1)F_{N:1}(t,d\phi_1dV_1)
\\
& \hspace{0.5cm} =\frac{d}{dt}\int_{(\bS\fH)^N\times C(\bbt^d)^N}\chi(\phi_1,V_1)F_N(t,d\Phi_Nd\cV_N) 
\\
& \hspace{0.5cm} =\int_{(\bS\fH)^N\times C(\bbt^d)^N}\la d\chi(\phi_1,V_1),X^f_N(t,\Phi_N,\cV_N)_1\ra F_N(t,d\Phi_Nd\cV_N) 
\\
&\hspace{0.7cm} +\int_{(\bS\fH)^N\times C(\bbt^d)^N}\la d\chi(\phi_1,V_1),X^i_N(t,\Phi_N,\cV_N)_1\ra F_N(t,d\Phi_Nd\cV_N) 
\\
& \hspace{0.5cm} =\int_{(\bS\fH)^N\times C(\bbt^d)^N}\la d_{\phi_1}\chi(\phi_1,V_1),Y^f(t,\phi_1,V_1)\ra F_N(t,d\Phi_Nd\cV_N) 
\\
&\hspace{0.7cm} +\int_{(\bS\fH)^N\times C(\bbt^d)^N}\la d_{\phi_1}\chi(\phi_1,V_1),Y^i_1(t,\Phi_N)\ra F_N(t,d\Phi_Nd\cV_N)\,.
\end{aligned}
\end{align}
The penultimate integral on the last right hand side of \eqref{C-3} is recast as follows:
\begin{align*}
\begin{aligned}
& \int_{(\bS\fH)^N\times C(\bbt^d)^N}\la d_{\phi_1}\chi(\phi_1,V_1),Y^f(t,\phi_1,V_1)\ra F_N(t,d\Phi_Nd\cV_N) 
\\
& \hspace{1cm} =
\int_{\bS\fH\times C(\bbt^d)}\la d_{\phi_1}\chi(\phi_1,V_1),Y^f(t,\phi_1,V_1)\ra F_{N:1}(t,d\phi_1dV_1).
\end{aligned}
\end{align*}
We also observe that 
\begin{align*}
\begin{aligned}
& \int_{(\bS\fH)^N\times C(\bbt^d)^N}\la d_{\phi_1}\chi(\phi_1,V_1),Y^i_1(t,\Phi_N)\ra F_N(t,d\Phi_Nd\cV_N) 
\\
& \hspace{0.5cm}  =\frac1N\sum_{l=1}^N\int_{(\bS\fH)^N\times C(\bbt^d)^N}\la d_{\phi_1}\chi(\phi_1,V_1),(I-\bP_{\phi_1})\phi_l\ra F_N(t,d\Phi_Nd\cV_N) 
\\
& \hspace{0.5cm}  =\frac1N\int_{(\bS\fH)^N\times C(\bbt^d)^N}\la d_{\phi_1}\chi(\phi_1,V_1),(I-\bP_{\phi_1})\phi_1\ra F_N(t,d\Phi_Nd\cV_N) 
\\
& \hspace{0.7cm} +\frac{N-1}N\int_{(\bS\fH)^N\times C(\bbt^d)^N}\la d_{\phi_1}\chi(\phi_1,V_1),(I-\bP_{\phi_1})\phi_2\ra F_N(t,d\Phi_Nd\cV_N) 
\\
& \hspace{0.5cm} =\frac{N-1}N\int_{(\bS\fH)^2\times C(\bbt^d)^2}\la d_{\phi_1}\chi(\phi_1,V_1),(I-\bP_{\phi_1})\phi_2\ra F_{N:2}(t,d\Phi_2d\cV_2),
\end{aligned}
\end{align*}
where we have used the relation:
\[ (I-\bP_{\phi_1})\phi_1=0. \]
Now, if $X$ is a $C_b$-vector field on $\bS\fH\times C(\bbt^d)$ of the form $X=Y\oplus 0$, we henceforth denote
\[ 
\Div_\phi(fY):=\Div(fX)\,. 
\]
Therefore, one has 
\begin{align*}
\begin{aligned}
& \frac{d}{dt}\int_{\bS\fH\times C(\bbt^d)}\chi(\phi_1,V_1)F_{N:1}(t,d\phi_1dV_1)
\\
& \hspace{0.5cm} =\int_{\bS\fH\times C(\bbt^d)}\la d_{\phi_1}\chi(\phi_1,V_1),Y^f(t,\phi_1,V_1)\ra F_{N:1}(t,d\phi_1dV_1) 
\\
&\hspace{0.7cm} +\frac{\kappa(N-1)}{2N}\int_{(\bS\fH)^2\times C(\bbt^d)^2}\la d_{\phi_1}\chi(\phi_1,V_1),(I-\bP_{\phi_1})\phi_2\ra F_{N:2}(t,d\Phi_2d\cV_2),
\end{aligned}
\end{align*}
for each $\chi\in C^1_b(\bS\fH\times C(\bbt^d))$, which is the weak formulation of
\[
\d_tF_{N:1}\!+\!\Div_{\phi_1}(F_{N:1}Y^f(t,\phi_1,V_1))\!+\!\tfrac{\kappa(N-1)}{2N}\Div_{\phi_1}\left((I\!-\!\bP_{\phi_1})\int_{\bS\fH\times C(\bbt^d)}\phi_2F_{N:2}d\phi_2dV_2\right)=0.
\]
At this point, we arrive at the mean-field kinetic equation by the following procedure. \newline

Assuming that $F_N=f^{\otimes N}$, we pass to the limit formally as $N\to\infty$ in both sides of the equation above to obtain
\be\lb{MFKinEq}
\begin{aligned}
& \d_tf +\Div_\phi\left(f\left(\tfrac1{i\hbar}U(\hbar t)^*(V+\om\star_x|U(\hbar t)\phi|^2)U(\hbar t)\phi\right)\right) 
\\
& \hspace{3cm} + \tfrac12\kappa \Div_\phi\left(f\int_{\bS\fH\times C(\bT^d)}(\tilde\phi-\la\phi|\tilde\phi\ra\phi)f(t,d\tilde\phi d\tilde V)\right)=0\,.
\end{aligned}
\ee
An alternate formulation of this equation is as follows: we set 
\[ 
\pi_1:\,\bS\fH\times C(\bbt^d)\ni(\phi,V)\mapsto\phi\in\bS\fH \quad\text{ and }\quad\rho_f(t):=\pi_1\#f(t). 
\]
Then, we have
\begin{align*}
\begin{aligned}
& \d_tf+\Div_\phi\left(f\left(\tfrac1{i\hbar}U(\hbar t)^*(V+\om\star_x|U(\hbar t)\phi|^2)U(\hbar t)\phi\right)\right)
\\
& \hspace{3cm} +\tfrac12 \kappa \Div_\phi\left(f(I-\bP_\phi)\int_{\bS\fH}\tilde\phi\rho_f(t,d\tilde\phi)\right)=0\,.
\end{aligned}
\end{align*}
Two remarks are in order, before we discuss the mathematical properties of \eqref{MFKinEq}.
\smallskip
In fact, one could have arrived at the equivalent formulation of the mean-field equation by a more intuitive and straightforward approach. Now, we return to the original SL model:
\[
\left\{
\ba
{}&{\mathrm i} \hbar\d_t\psi_k=(H_k+\om\star_x|\psi_k|^2)\psi_k+\tfrac12 {\mathrm i}\hbar \kappa (1-\bP_{\psi_k})\frac1N\sum_{l=1}^N\psi_l\,,
\\
&\psi_k\rstr_{t=0}=\psi_k^{0}, \quad k \in [N]\,.
\ea
\right.
\]
The key idea in the mean-field limit is to appeal to some form of the law of large number in order to replace the ensemble average
\[
\frac1N\sum_{l=1}^N\psi_l \qquad \hbox{ with } \quad \int_{\bS\fH}\tilde\psi\rho(t,d\tilde\psi), 
\]
where $\rho(t)$ is the probability distribution of $\phi$ on the unit sphere $\bS\fH$. This suggests replacing the SL equation above for the $k$-th oscillator
with
\[
{\mathrm i}\hbar\d_t\psi_k=(H_k+\om\star_x|\psi_k|^2)\psi_k+\tfrac12{\mathrm i}\hbar\kappa (1-\bP_{\psi_k})\int_{\bS\fH}\tilde\psi\rho(t,d\tilde\psi)\,.
\]
Then, the density $\rho(t,d\tilde\psi)$ is propagated by the evolution of $(\psi_1,\ldots,\psi_N)$ defined by the system of PDEs above. Since $H_k$ depends on $k$, it is impossible to define 
an evolution for $\phi_k$ independently of $k$, and this is the reason for considering the joint distribution of $\phi_k$ and $V_k$. This last system of equations for $k \in [N]$ can be viewed 
as the system of characteristics for the mean-field equation, in the same way as the original SL model is the system of characteristics of the $N$-oscillator Liouville equation.

\smallskip
Conjugating the evolution of the $N$-particle distribution by the free Schr\"odinger group as above might seem useless at first sight and requires perhaps some justification. In terms 
of the original wave functions $\psi_k$ instead of $\phi_k$ defined by $\phi_k(t,\cdot)=U(ht)^*\psi_k(t,\cdot)$, the mean-field equation would be recast as
\begin{align*}
\begin{aligned}
& \d_tg+\Div_\psi\left(g\left(\tfrac12i\hbar\Dlt_x\psi+\tfrac1{i\hbar}(V+\om\star_x|\psi|^2)\psi\right)\right)
\\
& \hspace{3cm} +\tfrac12 \kappa \Div_\psi\left(g(I-\bP_\psi)\int_{\bS\fH}\tilde\psi\rho_g(t,d\tilde\psi)\right)=0.
\end{aligned}
\end{align*}
There is an obvious difficulty with this formulation, which implicitly assumes that $g$ is concentrated on $\bS\fH\cap H^2(\bT^d)$. Indeed, $\mathrm{i} \Dlt_x\psi\in T_\psi\bS\fH$ only if $\psi\in H^2(\bbt^d)$. Since $U(t)$ is a strongly continuous unitary group on $\fH$, conjugating each equation in the SL system by $U(\hbar t)$ avoids this difficulty.

\subsection{Existence and uniqueness theory for the kinetic SL equation} \label{sec:3.3}


In this subsection, we prove the existence and uniqueness of the mean-field flow, i.e. the characteristic flow for \eqref{MFKinEq}. In what follows, the potential $V\in C(\bbt^d)$ and the Hartree kernel 
$\om\in L^\infty(\bbt^d)$ are given, together with the coupling constant $\ka$.

\begin{Thm}\lb{T-MFFlow}
Let $f^{0}\in {\mathcal P}_2(\mathbf S\mathfrak H\times C(\mathbb T^d))$ be given. Then, for each $\phi^{0}\in\bS\fH$, there exists a unique solution, denoted by $t\mapsto\phi(t|\phi^0)$, of the Cauchy problem:
\begin{equation}\label{DiffEq-phi}
\begin{cases}
\displaystyle \partial_t\phi(t|\phi^{0})=Y^f(t,\phi(t|\phi^{0}),V)+\tfrac\kappa2Y^{mf}[\rho_f](t,\phi(t|\phi^{0}))\,, \quad t > 0,
\\
\displaystyle \phi(0|\phi^{0})=\phi^{0}\,, \quad \rho_f(t)=\phi(t|\cdot)\#\rho_{f^{0}}\,,
\end{cases}
\end{equation}
where the vector field $Y^f$ is defined in \eqref{DefY}, and $Y^{mf}[\rho]$ is defined as follows.
\begin{equation} \label{New-1-1}
Y^{mf}[\rho](t,\phi):=(I-\bP_\phi)\int_{\bS\fH}\tilde\phi\rho(t,d{\tilde \phi})\,.
\end{equation}
\end{Thm}
\begin{proof} 
We split the proof into two steps. \newline

\noindent $\bullet$~Step A (Existence part):~Let $f^{0}\in\mathcal P_2(\mathbf S\mathfrak H\times C(\mathbb T^d))$ be given. Then, for each $\phi^{0}\in\mathbf S\mathfrak H$, we construct the sequence $\phi_n\equiv\phi_n(t|\phi^{0})\in\mathbf S\mathfrak H$ 
by the following prescription:
\begin{equation}\label{DiffEq-phin}
\left\{
\begin{aligned}
{}&\d_t\phi_n(t|\phi^{0})=Y^f(t,\phi_n(t|\phi^{0}),V ) \\
& \hspace{1.8cm} +\tfrac\kappa2\left(I-\mathbf P_{\phi_n(t|\phi^{0})}\right)\int_{\mathbf S\mathfrak H\times C(\mathbb T^d)}\phi_{n-1}(t|{\tilde \phi})f^{0}(d{\tilde \phi} dV)\,,\quad t >0,
\\
&\phi_n(0|\phi^{0})=\phi^{0}, \quad n \geq 1, \\
& \phi_0(t|\phi^{0}) \equiv \phi^{0}, \quad t \geq 0.
\end{aligned}
\right.
\end{equation}
That this differential system has a global solution $\bbr\ni t\mapsto\phi_n(t|\phi^{0})\in\bS\fH$ follows from the fact that
\[
\RE \Big \la Y^f(t,\phi_n(t|\phi^{0}),V)|\phi_n(t|\phi^{0}) \Big \ra=0\,,
\]
while
\[
\Big \la\phi_n(t|\phi^{0})|\left(I-\mathbf P_{\phi_n(t|\phi^{0})}\right)\xi \Big \ra=0\quad\text{ for all }\xi\in\fH\,.
\]
Hence the right-hand side of the system above defines continuous a vector field on $\bS\fH$. \newline

Next we verify that the vector field  is Lipschitz continuous uniformly in $t$. By the Cauchy-Schwarz inequality,
\[
\| |\phi|^2-|\psi|^2 \|_{L^1}=\|\phi(\overline\phi-\overline\psi)+(\phi-\psi)\overline\psi\|_{L^1}\le(\|\phi\|_\fH+\|\psi\|_\fH)\|\phi-\psi\|_\fH\,,
\]
so that
\[
\|Y^f(t,\phi,V)-Y^f(t,\tilde\phi,V)\|_\fH\le L_0\|\phi-{\tilde \phi}\|_\fH\,,
\]
with
\[
L_0:=\frac3\hbar(\|V\|_{L^\infty}+\|\om\|_{L^\infty})\,.
\]
Similar detailed calculations will be provided in \eqref{D-3-1} in a more general setting. On the other hand, we use $\|\phi\|_{\fH}=\|\psi\|_{\fH}=1$ to find 
\begin{align}\label{LipPphi}
\begin{aligned}
& \|(\mathbf P_\phi-\mathbf P_\psi)\xi\|_{\fH}  \le\|\phi-\psi\|_{\fH} |\langle\phi|\xi\rangle|+\|\psi\|_{\fH} |\langle\phi-\psi|\xi\rangle|  \\
& \hspace{2.4cm} \le \Big (\|\phi-\psi\|_{\fH} \|\phi\|_{\fH} +\|\psi\|_{\fH} \|\phi-\psi\|_{\fH} \Big )\|\xi\|_{\fH} \le 2\|\phi-\psi\|_{\fH} \|\xi\|_{\fH}.
\end{aligned}
\end{align}
This yields
\begin{equation} \label{New-2}
\left\|\left(\mathbf P_{\phi}-\mathbf P_{\psi}\right)\left(\int_{\bS\fH\times C(\mathbb T^d)} \psi f^{0}(d\psi dV)\right) \right\|_\fH\le 2\|\phi-\psi\|_\fH\,.
\end{equation}
These inequalities and the Cauchy-Lipschitz theorem in Appendix \ref{S-CL} --- see also Appendix \ref{SLiouv} --- imply the global existence of $\bbr\ni t\mapsto\phi_n(t|\phi^0)\in\bS\fH$
for each $n\ge 0$, once $\phi_{n-1}(~\cdot~|\phi^0)$ is known. In what follows, we consider the convergence of $\{ \phi_n \}$. It follows from \eqref{DiffEq-phin} that 
\begin{align*}
\begin{aligned}
& \phi_{n+1}(t|\phi^{0})-\phi_n(t|\phi^{0}) \\
& \hspace{0.5cm} =\int_0^t(Y^f(s,\phi_n(s|\phi^{0}),V)-Y^f(s,\phi_{n-1}(s|\phi^{0}),V))ds
\\
& \hspace{0.7cm} +\tfrac\kappa2\int_0^t\left(I-\mathbf P_{\phi_{n-1}(s|\phi^{0})}\right)\left(\int_{\mathbf S\mathfrak H\times C(\mathbb T^d)}(\phi_n(s|\psi)-\phi_{n-1}(s|\psi))f^{0}(d\psi dV)\right)ds
\\
& \hspace{0.7cm} +\tfrac\kappa2\int_0^t\left(\mathbf P_{\phi_{n-1}(s|\phi^{0})}-\mathbf P_{\phi_n(s|\phi^{0})}\right)\left(\int_{\mathbf S\mathfrak H\times C(\mathbb T^d)}\phi_n(s|\psi)f^{0}(d\psi dV)\right)ds.
\end{aligned}
\end{align*}
By \eqref{LipPphi} and 
\[
\begin{aligned}
\left \|\int_{\mathbf S\mathfrak H\times C(\mathbb T^d)}\!\phi_n(s|\psi)f^{0}(d\psi dV)\right \|_{\fH}\le&\int_{\mathbf S\mathfrak H\times C(\mathbb T^d)}\|\phi_n(s|\psi)\|_{\fH} f^{0}(d\psi dV)
\\
\le&\int_{\mathbf S\mathfrak H\times C(\mathbb T^d)}\!f^{0}(d\psi dV)\!=\!1\,,
\end{aligned}
\]
one has
\[
 \left\|\left(\mathbf P_{\phi_{n-1}(s|\phi^{0})}\!-\!\mathbf P_{\phi_n(s|\phi^{0})}\right)\left(\int_{\mathbf S\mathfrak H\times C(\mathbb T^d)}\phi_n(s|\psi)f^{0}(d\psi dV)\right)ds\right\|_{\fH}\!\le\!2\|\phi_{n-1}(s|\phi^{0})\!-\!\phi_n(s|\phi^{0})\|_{\fH}.
\]
Thus, we have
\begin{align*}
\begin{aligned}
& \|\phi_{n+1}(t|\phi^{0})-\phi_n(t|\phi^{0})\|_{\fH} \\
& \hspace{0.8cm} \le(L_1+\kappa)\left|\int_0^t\|\phi_n(t_1|\phi^{0})-\phi_{n-1}(t_1|\phi^{0})\|_{\fH} dt_1\right|
\\
& \hspace{1cm} +\tfrac\kappa2\Bigg|\int_0^t\underbrace{\left\|\int_{\mathbf S\mathfrak H\times C(\mathbb T^d)}(\phi_n(s|\psi)-\phi_{n-1}(s|\psi))f^0(d\psi dV)\right\|_{\fH}}_{\le\sup_{\|\psi\|_{\fH}=1}\|\phi_n(s|\psi)-\phi_{n-1}(s|\psi)\|_{\fH}}ds\Bigg|.
\end{aligned}
\end{align*}
We set 
\[
\eta_n(t):=\sup_{\|\phi^{0}\|_{\fH}=1}\|\phi_{n+1}(t|\phi^{0})-\phi_n(t|\phi^{0})\|_{\fH}.
\]
Then, one has
\begin{align*}
\begin{aligned}
\eta_n(t) &\le \Big ( L_1+ \frac{3 \kappa}{2} \Big )\left|\int_0^t\eta_{n-1}(t_1)dt_1\right| \leq (L_1+\tfrac{3\kappa}2)^2\left|\int_0^t\left|\int_0^{t_1}\eta_{n-2}(t_2)dt_2\right|dt_1\right| \\
&\le\ldots\le\frac{(L_1+\tfrac{3\kappa}2)^n|t|^n}{n!}\sup_{|s|\le|t|}\eta_0(s), \quad \mbox{for all}~n \geq 1.
\end{aligned}
\end{align*}
Besides, we have
\[
\eta_0(t)\le\left|\int_0^t\sup_{\|\phi^0\|=1}\|Y^f(s,\phi^{0},V)\|ds\right|\le\frac{|t|}{\hbar}(\|V\|_{L^\infty}+\|\omega\|_{L^\infty}),
\]
where we used the relation:
\[
\left(I-\mathbf P_{\phi^{0}}\right)\int_{\mathbf S\mathfrak H\times C(\mathbb T^d)}\phi^{0}f^{0}(d\psi dV)=\left(I-\mathbf P_{\phi^{0}}\right)\phi^{0}=0\,.
\]
Thus, we have
\[
\eta_n(t)\le\frac{(L_1+\tfrac{3\kappa}2)^n|t|^{n+1}}{n!\hbar}(\|V\|_{L^\infty}+\|\omega\|_{L^\infty})\,,
\]
so that the series
\[
\sum_{n\ge 0}\eta_n(t)<\infty
\]
for all $t\in \bbr$. Hence $\phi_n(t|\phi^{0})\to\phi(t|\phi^{0})$ uniformly in $t\in[-T,T]$ for all $T>0$ and in $\phi^{0}\in\mathbf S\mathfrak H$. In particular, $(t,\phi^{0})\mapsto\phi(t|\phi^{0})$ 
is continuous from $\bbr \times\mathbf S\mathfrak H$ to $\mathbf S\mathfrak H$ (as a uniform limit of a sequence of continuous functions). Passing to the limit in the equality defining 
$\phi_{n+1}$ in terms of $\phi_n$, we conclude that
\[
\begin{aligned}
\phi(t|\phi^{0})=\phi^{0}&+\int_0^tY^f(s,\phi(s|\phi^{0}),V)ds\!+\!\int_0^t\left(I-\mathbf P_{\phi(s|\phi^{0})}\right)\left(\int_{\mathbf S\mathfrak H\times C(\mathbb T^d)}\phi(s|\psi)f^{0}(d\psi dV)\right)ds
\\
=\phi^{0}&+\int_0^tY^f(s,\phi(s|\phi^{0}),V)ds+\tfrac\kappa2\int_0^tY^{mf}[\rho_f](s,\phi(s|\phi^{0}))ds\,,
\end{aligned}
\]
denoting $\rho_f(t,\cdot)=\phi(t,\cdot)_\#\rho_{f^{0}}$, where $\rho_{f^{0}}$ is the push-forward of the probability measure $f^{0}$ by the projection $(\psi,V)\mapsto\psi$. This equality implies that 
$t\mapsto\phi(t|\phi^{0})$ is of class $C^1$ on $\bbr$, and that
\[
\begin{cases}
\displaystyle \partial_t\phi(t|\phi^{0})=Y^f(t,\phi(t|\phi^{0}),V)+\tfrac\kappa2Y^{mf}[\rho_f](t,\phi(t|\phi^{0})), \quad t > 0, \\
\displaystyle \phi(0|\phi^{0})=\phi^{0}\,.
\end{cases}
\]
This completes the existence part. \newline

\noindent $\bullet$~Step B (Uniqueness part):~Assume that there exists another solution $(t,\phi^{0})\mapsto\tilde\phi(t|\phi^{0})$, continuous from $\bbr \times\mathbf S\mathfrak H$ to $\mathbf S\mathfrak H$ such that
\begin{align*}
\begin{aligned}
\tilde\phi(t|\phi^{0}) &=\phi^{0} +\int_0^tY^f(s,\tilde\phi(s|\phi^{0}),V)ds+\tfrac\kappa2\int_0^tY^{mf}[\rho_{\tilde f}](s,\tilde\phi(s|\phi^{0}))ds
\\
&=\phi^{0}\!+\!\int_0^tY^f(s,\tilde\phi(s|\phi^{0}),V)ds\!+\!\int_0^t\left(I\!-\!\mathbf P_{\tilde\phi(s|\phi^{0})}\right)\left(\int_{\mathbf S\mathfrak H\times C(\mathbb T^d)}\tilde\phi(s|\psi)f^{0}(d\psi dV)\right)ds\,,
\end{aligned}
\end{align*}
denoting $\rho_{\tilde f}(t,\cdot)=\tilde \phi(t,\cdot)\#\rho_{f^{0}}$. We use the same argument as in Step A to find that
\[
\sup_{\|\phi^{0}\|_{\fH}=1}\|\phi(t|\phi^{0})-\tilde\phi(t|\phi^{0})\|_{\fH} \le (L_1+\tfrac{3\kappa}2)\left|\int_0^t\sup_{\|\phi^{0}\|_{\fH}=1}\|\phi(s|\phi^{0})-\tilde\phi(s|\phi^{0})\|_{\fH} ds\right|.
\]
Then, we use Gronwall's lemma and  $\phi(0|\phi^{0}) = \tilde\phi(0|\phi^{0}) = \phi^0$ to see
\[
\sup_{\|\phi^{0}\|_{\fH}=1}\|\phi(t|\phi^{0})-\tilde\phi(t|\phi^{0})\|_{\fH}=0\,,\quad\text{ for all }t\in \bbr\,.
\]
This yields the uniqueness part:
\[ 
\phi(t|\phi^{0}) = \tilde\phi(t|\phi^{0}) \quad \mbox{in}~\fH,
\]
for each $t > 0$.
\end{proof}


\section{Mean-field limit: Fluctuation estimate} \label{sec:4}
\setcounter{equation}{0} 
In this section, we present a quantatitive fluctuation estimate from the SL model to the kinetic SL equation which has been derived in the previous section via the BBGKY hierarchy argument. 


\subsection{Dynamics of couplings}


We recall the definition \eqref{Def-SN(tt0)} of the flow map $\mathcal S_N(t,t_0)$: for all $\phi_1^0,\ldots,\phi_N^0\in\mathbf S\mathfrak H$ and all $V_1,\ldots,V_N\in C(\mathbb T^d)$,
\[
\mathcal S_N(t,t_0)(\phi_1^0,\ldots,\phi_N^0,V_1,\ldots,V_N)=(\phi_1(t),\ldots,\phi_N(t),V_1,\ldots,V_N),
\]
where $t\mapsto(\phi_1(t),\ldots,\phi_N(t))$ is the solution to the system of differential equations:
\[
\left\{
\begin{aligned}
{}&\partial_t\phi_k=Y^f(t,\phi_k,V_k)+\tfrac12\kappa Y^i_k(t,\Phi_N)\,,\quad t > t_0, \quad  k \in [N],
\\
&\phi_k(t_0)=\phi_k^0\in\mathbf S\mathfrak H\,,
\end{aligned}
\right.
\]
where we recall that $\Phi_N(t)=(\phi_1(t),\ldots,\phi_N(t))$. The vector fields  $Y^f$ and $Y^i$ have been defined in the preceding section. Along with these vector fields, we also consider the mean-field vector field:
\[ 
Y^{mf}[\rho_f](t,\phi):=\int_{\bS\fH}(\tilde\phi-\la\phi|\tilde\phi\ra\phi)\rho_f(t,d\tilde\phi)=(I-\bP_\phi)\int_{\bS\fH}\tilde\phi\rho_f(t,d\tilde\phi)\,.
\]
With given $\rho_f$, we consider the differential equation:
\[
\left\{
\begin{aligned}
{}&\partial_t\tilde\phi=Y^f(t,\tilde\phi,V)+\tfrac12\kappa Y^{mf}[\rho_f](t,\tilde\phi)\,,
\\
&\tilde\phi(t_0)=\tilde\phi^0\in\mathbf S\mathfrak H\,,
\end{aligned}
\right.
\]
whose solution is denoted by 
\begin{equation}\label{Def-s(tt0)}
\tilde\phi(t)=:\mathfrak s(t,t_0)\tilde\phi^0\,.
\end{equation}
Finally, we consider the map:
\begin{align*}\label{Def-TN(tt0)}
\begin{aligned}
& \mathcal T_N(t,t_0)(\phi_1^0,\ldots,\phi_N^0,V_1,\ldots,V_N;\tilde\phi_1^0,\ldots,\tilde\phi_N^0,\tilde V_1,\ldots,\tilde V_N)
\\
& \hspace{0.5cm} =(\mathcal S_N(t,t_0)(\phi_1^0,\ldots,\phi_N^0,V_1,\ldots,V_N);\mathfrak s(t,t_0)\tilde\phi^0_1,\ldots,\mathfrak s(t,t_0)\tilde\phi^0_N,\tilde V_1,\ldots,\tilde V_N).
\end{aligned}
\end{align*}

\begin{Lem}\label{L4.1}
Let $Q^{0}_N$ be a coupling of $F^{0}_N$ and $(f^{0})^{\otimes N}$, and we set                                                                                                                                                                                                                                                                                                                                                                                                                                                                                                                                                                                                             
\[
Q_N(t)=\mathcal T_N(t,0)\#Q_N^0\,.
\]
Then, the following assertions hold.
\begin{enumerate}
\item
For each $t\ge 0$, the probability measure $Q_N(t)$ on $(\mathbf S\mathfrak H)^N\!\times\!C(\mathbb T^d)^N\!\times\!(\mathbf S\mathfrak H)^N\!\times\!C(\mathbb T^d)^N$ is a coupling
of $F_N(t)$ and $f(t)^{\otimes N}$.
\vspace{0.1cm}
\item
If $Q^{0}_N$ is a symmetric coupling of $F^{0}_N$ and $(f^{0})^{\otimes N}$, then $Q_N(t)$ is a symmetric coupling of $F_N(t)$ and $f(t)^{\otimes N}$.
\vspace{0.1cm}
\item
The time-dependent probability measure $Q_N(t)$ on $(\mathbf S\mathfrak H)^N\!\times\!C(\mathbb T^d)^N\!\times\!(\mathbf S\mathfrak H)^N\!\times\!C(\mathbb T^d)^N$ is a Lagrangian
weak solution of
\[
\begin{cases}
\displaystyle \d_tQ_N+ \sum_{k=1}^N\Div_{\phi_k}(Q_N(Y^f(t,\phi_k,V_k)+\tfrac12\kappa Y^i_k(t,\Phi_N))) \\
\displaystyle \hspace{0.5cm} +\sum_{k=1}^N\Div_{\tilde\phi_k}(Q_N(Y^f(t,\tilde\phi_k,\tilde V_k)+\tfrac12 \kappa Y^{mf}[\rho_f](t,\tilde\phi_k)))=0, \\
\displaystyle Q_N\rstr_{t=0}=Q^{0}_N.
\end{cases}
\]
\end{enumerate}
\end{Lem}

\begin{proof} 
(i)~Let $\chi\in C_b((\mathbf S\mathfrak H)^N\times C(\mathbb T^d)^N)$; then we have, with the notation \eqref{NotPhiNVN},
\begin{align*}
\begin{aligned}
& \int_{(\mathbf S\mathfrak H)^N\times C(\mathbb T^d)^N}\chi(\Phi_N,\mathcal V_N)Q_N(t,d\Phi_Nd\mathcal V_Nd\tilde\Phi_Nd\tilde{\mathcal V}_N)
\\
&  \hspace{0.5cm} =\!\int_{(\mathbf S\mathfrak H)^N\times C(\mathbb T^d)^N}\!\chi(\mathcal S_N(t,0)\!\cdot\!(\Phi_N,\mathcal V_N))Q^0_N(d\Phi_Nd\mathcal V_Nd\tilde\Phi_Nd\tilde{\mathcal V}_N)
\\
&\hspace{0.5cm} =\!\int_{(\mathbf S\mathfrak H)^N\times C(\mathbb T^d)^N}\!\chi(\mathcal S_N(t,0)\!\cdot\!(\Phi_N,\mathcal V_N))F_N^0(d\Phi_Nd\mathcal V_N)
\\
&\hspace{0.5cm} =\int_{(\mathbf S\mathfrak H)^N\times C(\mathbb T^d)^N}\chi(\Phi_N,\mathcal V_N)F_N(t,d\Phi_Nd\mathcal V_N).
\end{aligned}
\end{align*}
Similarly, we have 
\begin{align*}
\begin{aligned}
& \int_{(\mathbf S\mathfrak H)^N\times C(\mathbb T^d)^N}\chi(\tilde\Phi_N,\tilde{\mathcal V_N})Q_N(t,d\Phi_Nd\mathcal V_Nd\tilde\Phi_Nd\tilde{\mathcal V}_N) 
\\
&  \hspace{0.5cm} =\int_{(\mathbf S\mathfrak H)^N\times C(\mathbb T^d)^N}\chi(\mathfrak s(t,0)\tilde\phi_1,\ldots,\mathfrak s(t,0)\tilde\phi_N,\mathcal V_N)
Q_N^0(d\Phi_Nd\mathcal V_Nd\tilde\Phi_Nd\tilde{\mathcal V}_N)
\\
& \hspace{0.5cm} =\int_{(\mathbf S\mathfrak H)^N\times C(\mathbb T^d)^N}\chi(\mathfrak s(t,0)\tilde\phi_1,\ldots,\mathfrak s(t,0)\tilde\phi_N,\tilde V_1,\ldots,\tilde V_N)f^0(d\tilde\phi_1d\tilde V_1)\ldots f^0(d\tilde\phi_Nd\tilde V_N)
\\
&  \hspace{0.5cm} =\int_{(\mathbf S\mathfrak H)^N\times C(\mathbb T^d)^N}\chi(\tilde\Phi_N,\tilde{\mathcal V}_N)f_N(t,d\tilde\phi_1d\tilde V_1)\ldots f_N(t,d\tilde\phi_Nd\tilde V_N).
\end{aligned}
\end{align*}
These identities imply the first assertion.  \newline

\noindent (ii)~For each permutation $\sigma$ in the symmetric group $\mathfrak S_N$, we set 
\begin{align*}
\begin{aligned}
& \mathbf T_\sigma(\phi_1,\ldots,\phi_N,V_1,\ldots,V_N;\tilde\phi_1,\ldots,\tilde\phi_N,\tilde V_1,\ldots,\tilde V_N)
\\
& \hspace{1cm}  :=(T_\sigma(\phi_1,\ldots,\phi_N,V_1,\ldots,V_N);T_\sigma(\tilde\phi_1,\ldots,\tilde\phi_N,\tilde V_1,\ldots,\tilde V_N))
\\
& \hspace{1.1cm}  =(\phi_{\sigma(1)},\ldots,\phi_{\sigma(N)},V_{\sigma(1)},\ldots,V_{\sigma(N)};\tilde\phi_{\sigma(1)},\ldots,\tilde\phi_{\sigma(N)},\tilde V_{\sigma(1)},\ldots,\tilde V_{\sigma(N)}).
\end{aligned}
\end{align*}
Then, for each $\chi\in C_b((\mathbf S\mathfrak H)^N\times C(\mathbb T^d))$, each $t\ge 0$ and each $\sigma\in\mathfrak S_N$, one has
\begin{align*}
\begin{aligned}
& \int_{(\mathbf S\mathfrak H)^N\times C(\mathbb T^d)^N}\chi(\Phi_N,\mathcal V_N,\tilde\Phi_N,\tilde{\mathcal V}_N)\mathbf T_\sigma \#Q_N(t,d\Phi_Nd\mathcal V_Nd\tilde\Phi_Nd\tilde{\mathcal V}_N)
\\
&\hspace{0.2cm} =\int_{(\mathbf S\mathfrak H)^N\times C(\mathbb T^d)^N}\chi(\mathcal S_N(t,0)\circ T_\sigma(\Phi_N,\mathcal V_N),
\mathfrak s(t,0)\tilde\phi_{\sigma(1)},\ldots,\mathfrak s(t,0)\tilde\phi_{\sigma(N)},\tilde V_{\sigma(1)},\ldots,\tilde V_{\sigma(N)})
\\
& \hspace{3cm} \times Q_N^0(d\Phi_Nd\mathcal V_Nd\tilde\Phi_Nd\tilde{\mathcal V}_N)
\\
&\hspace{0.2cm} =\int_{(\mathbf S\mathfrak H)^N\times C(\mathbb T^d)^N}\chi(T_\sigma\circ\mathcal S_N(t,0)(\Phi_N,\mathcal V_N),
\mathfrak s(t,0)\tilde\phi_{\sigma(1)},\ldots,\mathfrak s(t,0)\tilde\phi_{\sigma(N)},\tilde V_{\sigma(1)},\ldots,\tilde V_{\sigma(N)})
\\
& \hspace{3cm} \times Q_N^0(d\Phi_Nd\mathcal V_Nd\tilde\Phi_Nd\tilde{\mathcal V}_N)
\\
& \hspace{0.2cm} =\int_{(\mathbf S\mathfrak H)^N\times C(\mathbb T^d)^N}\chi(\mathcal S_N(t,0)(\Phi_N,\mathcal V_N),\mathfrak s(t,0)\tilde\phi_1,\ldots,\mathfrak s(t,0)\tilde\phi_N,\tilde{\mathcal V}_N)
Q_N^0(d\Phi_Nd\mathcal V_Nd\tilde\Phi_Nd\tilde{\mathcal V}_N)
\\
& \hspace{0.2cm}=\int_{(\mathbf S\mathfrak H)^N\times C(\mathbb T^d)^N}\chi(\Phi_N,\mathcal V_N;\tilde\phi_N,\tilde{\mathcal V}_N)Q_N(t,d\Phi_Nd\mathcal V_Nd\tilde\Phi_Nd\tilde{\mathcal V}_N)\,.
\end{aligned}
\end{align*}
The second equality follows from \eqref{New-1}, while the third follows from the fact that $Q_N^0$ is a symmetric coupling of $F_N^0$ and $(f^0)^{\otimes N}$. This chain of equalities proves that
\[
\mathbf T_\sigma \#Q_N(t)=Q_N(t)\,,\quad\text{ for each }\sigma\in\mathfrak S_N.
\]
This is precisely the second assertion. \newline

\noindent (ii)~Finally, the Liouville equation in (3) follows from the definitions \eqref{Def-SN(tt0)} of $\mathcal S_N(t,0)$, and \eqref{Def-s(tt0)} of $\mathfrak s(t,0)$, and from Lemma \ref{L-B.1}  in Appendix \ref{SLiouv}.
\end{proof}

Before we estimate the 2-Wasserstein distance between one-marginal distribution and one-oscillator distribution, we provide a bound for the functional $D_N$ defined by the following formula: for two solutions 
$\{ (\phi_k, V_k) \}$  and $\{ ({\tilde \phi}_k, {\tilde V}_k) \}$ of \eqref{C-1}, 
\begin{equation*} \label{D-0}
D_N(t):=\frac1{2N}\sum_{k=1}^N\int_{\Lambda_N\times\Lambda_N}\Big(\|\phi_k-\tilde\phi_k\|^2_\fH+\|V_k-\tilde V_k\|^2_{L^\infty} \Big)Q_N(t,d\Phi_Nd\cV_Nd\tilde\Phi_Nd\tilde\cV_N)
\end{equation*}
for each $t\ge 0$. Here, $\Lambda_N$ denotes the  $N$-oscillator phase space:
\[ 
\Lambda_N := (\mathbf S\mathfrak H)^N\times C(\mathbb T^d)^N. \
\]
We shall henceforth assume that
\[
\int_{\Lambda_N\times\Lambda_N}\Big(\|V_k\|^2_{L^\infty}+\|\tilde V_k\|^2_{L^\infty} \Big )Q^{0}_N(d\Phi_Nd\cV_Nd\tilde\Phi_Nd\tilde\cV_N)<\infty\,.
\]
In the next subsection, we will derive a differential inequality for $D_N$.

\subsection{Evolution of $D_N(t)$} 

By definition of the divergence, we have
\begin{align}
\begin{aligned} \label{D-1}
\dot D_N(t)&\!=\!\tfrac1N\sum_{k=1}^N\int_{\Lambda_N\times\Lambda_N}\RE \Big \la(Y^f(t,\phi_k,V_k)\!-\!Y^f(t,\tilde\phi_k,\tilde V_k)) \Big |\phi_k\!-\!\tilde\phi_k \Big \ra Q_N(t,d\Phi_Nd\cV_Nd\tilde\Phi_Nd\tilde\cV_N) 
\\
&+\!\tfrac{\kappa}{2N}\sum_{k=1}^N\int_{\Lambda_N\times\Lambda_N}\RE \Big \la(Y^i_k(t,\Phi_N)\!-\!Y^{mf}[\rho_f](t,\tilde\phi_k)) \Big |\phi_k\!-\!\tilde\phi_k \Big \ra Q_N(t, d\Phi_Nd\cV_Nd\tilde\Phi_Nd\tilde\cV_N) \\
&=: {\mathcal I}_{21} + \frac{\kappa}{2} {\mathcal I}_{22}.
\end{aligned}
\end{align}
In the following lemma, we estimate the terms ${\mathcal I}_{2i}$ one by one.
\begin{Lem} \label{L4.2}
Let  $\{ (\phi_k, V_k) \}$ and $\{ ({\tilde \phi}_k, {\tilde V}_k) \}$ be solutions to \eqref{C-1}. Then, one has
\begin{eqnarray*}
&& (i)~|{\mathcal I}_{21} | \leq \frac{L_1}{2N} \sum_{k=1}^N\int_{\Lambda_N\times\Lambda_N}\Big (3\|\phi_k-\tilde\phi_k\|^2_\fH+\|V_k-\tilde V_k\|^2_{L^{\infty}} \Big)  Q_N(t,d\Phi_Nd\cV_Nd\tilde\Phi_Nd\tilde\cV_N). \cr
&& (ii)~|{\mathcal I}_{22} | \leq \frac7{2N}\sum_{k=1}^N \int_{\Lambda_N\times\Lambda_N}\|\phi_k-\tilde\phi_k\|^2_\fH Q_N(t,d\Phi_Nd\cV_Nd\tilde\Phi_Nd\tilde\cV_N)  \\
&& \hspace{1.5cm} + \tfrac12 \int_{\Lambda_N\times\Lambda_N}\left\|\frac1N\sum_{l=1}^N\left(\tilde\phi_l-\int_{\bS\fH}\tilde\phi\rho_f(t,d\tilde\phi)\right)\right\|^2_\fH Q_N(t,d\Phi_Nd\cV_Nd\tilde\Phi_Nd\tilde\cV_N),
\end{eqnarray*}
where $L_1$ is a positive constant defined by the following relation:
\begin{equation} \label{D-1-1}
L_1:=\tfrac1\hbar\max\left \{ \max_{1\le k\le N}\|V_k\|_{L^\infty}+3\|\om\|_{L^\infty},1\right \}.  
\end{equation}
\end{Lem}
\begin{proof}
\noindent (i)~First, we note that 
\begin{align*}
\begin{aligned}
& Y^f(t,\phi_k,V_k)-Y^f(t,\tilde\phi_k,\tilde V_k) \\
& \hspace{0.5cm} = \tfrac1{{\mathrm i}\hbar}U(\hbar t)^*(V_k+\om\star_x|U(\hbar t)\phi_k|^2)U(\hbar t)(\phi_k-\tilde\phi_k)
+ \tfrac1{{\mathrm i}\hbar}U(\hbar t)^*(V_k-\tilde V_k)U(\hbar t)\tilde\phi_k
\\
& \hspace{0.7cm} + \tfrac1{{\mathrm i}\hbar}U(\hbar t)^* \Big(\om\star_x(|U(\hbar t)\phi_k|^2-|U(\hbar t)\tilde\phi_k|^2) \Big)U(\hbar t)\tilde\phi_k.
\end{aligned}
\end{align*}
This yields
\begin{align}
\begin{aligned} \label{D-2}
& \|Y^f(t,\phi_k,V_k)-Y^f(t,\tilde\phi_k,\tilde V_k)\|_\fH \\
& \hspace{0.8cm} \le
\tfrac1{\hbar}\left(\|V_k\|_{L^\infty}+\|\om\|_{L^\infty}\right)\|\phi_k-\tilde\phi_k\|_\fH
+ \tfrac1\hbar\|V_k-\tilde V_k\|_{L^\infty}  \\
&\hspace{1cm} +
\|\om\star_x(|U(\hbar t)\phi_k|^2-|U(\hbar t)\tilde\phi_k|^2)\|_{L^\infty}  \\
& \hspace{0.8cm} \le \tfrac1\hbar\left(\max_{1\le k\le N}\|V_k\|_{L^\infty}+\|\om\|_{L^\infty}\right)\|\phi_k-\tilde\phi_k\|_\fH+\tfrac1\hbar\|V_k-\tilde V_k\|_{L^\infty}  \\
& \hspace{1cm}+\tfrac1\hbar\|\om\|_{L^\infty}\||U(\hbar t)\phi_k|^2-|U(\hbar t)\tilde\phi_k|^2\|_{L^1}.
\end{aligned}
\end{align}
Now, we use  the identity $|a|^2-|b|^2=(\bar a-\bar b)a+\bar b(a-b)$ and the Cauchy-Schwarz inequality to see that 
\begin{equation} \label{D-3}
\||U(\hbar t)\phi_k|^2-|U(\hbar t)\tilde\phi_k|^2\|_{L^1}\le 2\|\phi_k-\tilde\phi_k\|_\fH.
\end{equation}
Then, we combine \eqref{D-2} and \eqref{D-3} to get 
\begin{align}
\begin{aligned} \label{D-3-1}
& \|Y^f(t,\phi_k,V_k)-Y^f(t,\tilde\phi_k,\tilde V_k)\|_\fH \\
& \hspace{1cm} \le\tfrac1\hbar\left(\max_{1\le k\le N}\|V_k\|_{L^\infty}+3\|\om\|_{L^\infty}\right)\|\phi_k-\tilde\phi_k\|_\fH+\tfrac1\hbar\|V_k-\tilde V_k\|_{L^\infty} \\
& \hspace{1cm} \leq L_1 (\|\phi_k-\tilde\phi_k\|_\fH+\|V_k-\tilde V_k\|_{L^{\infty}})\,,
\end{aligned}
\end{align}
where $L_1$ is the positive constant defined in \eqref{D-1-1}. \newline

\noindent Finally, in \eqref{D-1}, we use \eqref{D-3-1} and the trivial inequality $|ab| \leq \frac{a^2}{2} + \frac{b^2}{2}$ to obtain
\begin{align*}
\begin{aligned}
|{\mathcal I}_{21} | &\le\frac{L_1}N\sum_{k=1}^N\int_{\Lambda_N\times\Lambda_N}(\|\phi_k-\tilde\phi_k\|_\fH+\|V_k-\tilde V_k\|_{L^{\infty}})\|\phi_k-\tilde\phi_k\|_\fH Q_N(t,d\Phi_Nd\cV_Nd\tilde\Phi_Nd\tilde\cV_N) \\
&\leq \frac{L_1}{2N} \sum_{k=1}^N  \int_{\Lambda_N\times\Lambda_N}\Big(3 \|\phi_k-\tilde\phi_k\|_\fH^2 + |V_k-\tilde V_k\|_{L^{\infty}}^2\Big) Q_N(t,d\Phi_Nd\cV_Nd\tilde\Phi_Nd\tilde\cV_N).
\end{aligned}
\end{align*}

\vspace{0.2cm}

\noindent (ii)~For the interaction part, we observe that
\begin{align}
\begin{aligned} \label{D-3-2}
& Y^i_k(t,\Phi_N)-Y^{mf}[\rho_f](t,\tilde\phi_k)  \\
& \hspace{1cm} =(I-\bP_{\phi_k})\frac1N\sum_{l=1}^N\phi_l-(I-\bP_{\tilde\phi_k})\int_{\bS\fH}\tilde\phi\rho_f(t,d\tilde\phi) \\
& \hspace{1cm} =(\bP_{\tilde\phi_k}-\bP_{\phi_k})\frac1N\sum_{l=1}^N\phi_l+(I-\bP_{\tilde\phi_k})\frac1N\sum_{l=1}^N(\phi_l-\tilde\phi_l) \\
& \hspace{1.2cm} +(I-\bP_{\tilde\phi_k})\frac1N\sum_{l=1}^N\left(\tilde\phi_l-\int_{\bS\fH}\tilde\phi\rho_f(t,d\tilde\phi)\right).
\end{aligned}
\end{align}
Then in \eqref{D-3-2}, we use the elementary computations:
\[
\| (I - \bP_{\tilde\phi_k}) \xi \|_{\fH} \leq \|\xi \|_{\fH} \quad \mbox{and} \quad  \|(\bP_{\tilde\phi}-\bP_{\phi})\xi\|_\fH=\|\la\tilde\phi-\phi|\xi\ra\tilde\phi+\la\phi|\xi\ra(\tilde\phi-\phi)\|_\fH\le 2\|\tilde\phi-\phi\|_\fH\|\xi\|_\fH\,,
\]
together with the Cauchy-Schwarz inequality to see that
\begin{align*}
\begin{aligned}  
&\Big | \Big \la(Y^i_k(t,\Phi_N)-Y^{mf}[\rho_f](t,\tilde\phi_k))|\phi_k-\tilde\phi_k \Big \ra \Big|  \\
& \hspace{0.5cm} \le 2\|\phi_k-\tilde\phi_k\|^2_\fH\left\|\frac1N\sum_{l=1}^N\phi_l\right\|_\fH  +\frac1N\sum_{l=1}^N\|\phi_l-\tilde\phi_l\|_\fH\|\phi_k-\tilde\phi_k\|_\fH \\
& \hspace{0.7cm} +\left\|\frac1N\sum_{l=1}^N\left(\tilde\phi_l-\int_{\bS\fH}\tilde\phi\rho_f(t,d\tilde\phi)\right)\right\|_\fH\|\phi_k-\tilde\phi_k\|_\fH\,.
\end{aligned}
\end{align*}
Therefore, we have 
\begin{align*}
\begin{aligned} \label{D-3-3}
& \left|\frac1N\sum_{k=1}^N \Big \la(Y^i_k(t,\Phi_N)-Y^{mf}[\rho_f](t,\tilde\phi_k)) \Big | \phi_k-\tilde\phi_k \Big \ra\right| \\
& \hspace{1cm} \le\frac2N\sum_{k=1}^N\|\phi_k-\tilde\phi_k\|^2_\fH+\frac1{N^2}\sum_{k,l=1}^N\|\phi_l-\tilde\phi_l\|_\fH\|\phi_k-\tilde\phi_k\|_\fH \\
& \hspace{1cm} +\tfrac12\left\|\frac1N\sum_{l=1}^N\left(\tilde\phi_l-\int_{\bS\fH}\tilde\phi\rho_f(t,d\tilde\phi)\right)\right\|^2_\fH+\frac1{2N}\sum_{k=1}^N\|\phi_k-\tilde\phi_k\|^2_\fH \\
& \hspace{1cm} \le\frac5{2N}\sum_{k=1}^N\|\phi_k-\tilde\phi_k\|^2_\fH+\frac1{N^2}\sum_{k,l=1}^N\tfrac12 \Big (\|\phi_l-\tilde\phi_l\|^2_\fH+\|\phi_k-\tilde\phi_k\|^2_\fH \Big ) \\
& \hspace{1cm} +\tfrac12\left\|\frac1N\sum_{l=1}^N\left(\tilde\phi_l-\int_{\bS\fH}\tilde\phi\rho_f(t,d\tilde\phi)\right)\right\|^2_\fH
\\
& \hspace{1cm} \le\frac7{2N}\sum_{k=1}^N\|\phi_k-\tilde\phi_k\|^2_\fH+\tfrac12\left\|\frac1N\sum_{l=1}^N\left(\tilde\phi_l-\int_{\bS\fH}\tilde\phi\rho_f(t,d\tilde\phi)\right)\right\|^2_\fH,
\end{aligned}
\end{align*}
where we have used the bound:
\[  \left\|\frac1N\sum_{l=1}^N\phi_l\right\|_\fH \leq \frac{1}{N} \sum_{l=1}^{N} \|\phi_l \|_{\fH} = 1.  \]
This yields the desired estimate:
\begin{align*}
\begin{aligned}
|{\mathcal I}_{22} | &\leq \frac7{2N}\sum_{k=1}^N \int_{\Lambda_N\times\Lambda_N}\|\phi_k-\tilde\phi_k\|^2_\fH Q_N(t,d\Phi_Nd\cV_Nd\tilde\Phi_Nd\tilde\cV_N)  \\
&\hspace{0.2cm} + \tfrac12 \int_{\Lambda_N\times\Lambda_N}\left\|\frac1N\sum_{l=1}^N\left(\tilde\phi_l-\int_{\bS\fH}\tilde\phi\rho_f(t,d\tilde\phi)\right)\right\|^2_\fH Q_N(t,d\Phi_Nd\cV_Nd\tilde\Phi_Nd\tilde\cV_N).
\end{aligned}
\end{align*}
\end{proof}
Finally, we combine \eqref{D-1} and estimate in Lemma \ref{L4.3} to get 
\begin{align}
\begin{aligned} \label{D-3-4}
&\dot{D}_N(t) \le |{\mathcal I}_{21} | + \frac{\kappa}{2} |{\mathcal I}_{22} | \\
&\leq \frac{L_1}{2N} \sum_{k=1}^N\int_{\Lambda_N\times\Lambda_N}\Big (3\|\phi_k-\tilde\phi_k\|^2_\fH+\|V_k-\tilde V_k\|^2_{L^{\infty}} \Big)  Q_N(t,d\Phi_Nd\cV_Nd\tilde\Phi_Nd\tilde\cV_N) 
\\
& \hspace{0.2cm} +  \frac{7\kappa}{4N}\sum_{k=1}^N \int_{\Lambda_N\times\Lambda_N}\|\phi_k-\tilde\phi_k\|^2_\fH Q_N(t,d\Phi_Nd\cV_Nd\tilde\Phi_Nd\tilde\cV_N)  
\\
& \hspace{0.2cm} + \frac{\kappa}{4}\int_{\Lambda_N\times\Lambda_N}\left\|\frac1N\sum_{l=1}^N\left(\tilde\phi_l-\int_{\bS\fH}\tilde\phi\rho_f(t,d\tilde\phi)\right)\right\|^2_\fH Q_N(t,d\Phi_Nd\cV_Nd\tilde\Phi_Nd\tilde\cV_N) 
\\
&\leq \Big(3L_1+\frac{7\kappa}{2} \Big)D_N(t) +\tfrac14\kappa \int_{(\bS\fH)^N}\left\|\frac1N\sum_{l=1}^N\left(\tilde\phi_l-\int_{\bS\fH}\tilde\phi\rho_f(t,d\tilde\phi)\right)\right\|^2_\fH\rho^{\otimes N}_f(t,d\tilde\Phi_N)\,.
\end{aligned}
\end{align}
Note that the last integral on the right-hand side above involves the measure $\rho^{\otimes N}_f(t,d\tilde\Phi_N)$ (see Lemma \ref{L4.1} (1)).

\subsection{The poor man's law of large numbers}

In order to derive Gronwall's inequality from \eqref{D-3-4}, we need to estimate the term
\begin{align*}
\begin{aligned}
& \int_{(\bS\fH)^N}\left\|\frac1N\sum_{l=1}^N\left(\tilde\phi_l-\int_{\bS\fH}\tilde\phi\rho_f(t,d\tilde\phi)\right)\right\|^2_\fH\rho^{\otimes N}_f(t,d\tilde\Phi_N) \\
& \hspace{1cm} = \frac1{N^2}\sum_{k,l=1}^N\int_{(\bS\fH)^N}\La\tilde\phi_k\!-\!\int_{\bS\fH}\tilde\phi\rho_f(t,d\tilde\phi)\bigg|\tilde\phi_l\!-\!\int_{\bS\fH}\tilde\phi\rho_f(t,d\tilde\phi)\Ra\prod_{m=1}^N\rho_f(t,d\tilde\phi_m).
\end{aligned}
\end{align*}
Note that for $k\not=l$, we have
\begin{align*}
\begin{aligned}
& \int_{(\bS\fH)^N}\La\tilde\phi_k-\int_{\bS\fH}\tilde\phi\rho_f(t,d\tilde\phi)\bigg|\tilde\phi_l-\int_{\bS\fH}\tilde\phi\rho_f(t,d\tilde\phi)\Ra\prod_{m=1}^N\rho_f(t,d\tilde\phi_m) \\
& \hspace{0.5cm} =\!\La\int_{\bS\fH}\!\left(\tilde\phi_k\!-\!\!\int_{\bS\fH}\tilde\phi\rho_f(t,d\tilde\phi)\right)\rho_f(t,d\tilde\phi_k)\bigg|\int_{\bS\fH}\!\left(\tilde\phi_l\!-\!\!\int_{\bS\fH}\tilde\phi\rho_f(t,d\tilde\phi)\right)\rho_f(t,d\tilde\phi_l)\Ra\!=\!0.
\end{aligned}
\end{align*}
Therefore, only the terms with $k=l$ remain in 
\[
\int_{(\bS\fH)^N}\left\|\frac1N\sum_{l=1}^N\left(\tilde\phi_l-\int_{\bS\fH}\tilde\phi\rho_f(t,d\tilde\phi)\right)\right\|^2_\fH\rho^{\otimes N}_f(t,d\tilde\Phi_N)\,,
\]
and there are exactly $N$ such terms. Thus, one has
\begin{align}
\begin{aligned} \label{D-3-5}
& \int_{(\bS\fH)^N}\left\|\frac1N\sum_{l=1}^N\left(\tilde\phi_l-\int_{\bS\fH}\tilde\phi\rho_f(t,d\tilde\phi)\right)\right\|^2_\fH\rho^{\otimes N}_f(t,d\tilde\Phi_N) \\
& \hspace{1cm} \le\frac1{N^2}\sum_{k=1}^N\int_{(\bS\fH)^N}\left\|\tilde\phi_k-\int_{\bS\fH}\tilde\phi\rho_f(t,d\tilde\phi)\right\|^2_\fH\rho^{\otimes N}_f(t,d\tilde\Phi_N) \\
& \hspace{1cm} \le\frac1{N^2}\sum_{k=1}^N\int_{\bS\fH}\left\|\tilde\phi-\int_{\bS\fH}\tilde\phi\rho_f(t,d\tilde\phi)\right\|^2_\fH\rho_f(t,d\tilde\phi)
\\
& \hspace{1cm} \le\frac1N\int_{\bS\fH}\left\|\tilde\phi-\int_{\bS\fH}\tilde\phi\rho_f(t,d\tilde\phi)\right\|^2_\fH\rho_f(t,d\tilde\phi)\le\frac4N\,.
\end{aligned}
\end{align}
We have used the following relation in the last inequality:
\[
\left\|\tilde\phi-\int_{\bS\fH}\tilde\phi\rho_f(t,d\tilde\phi)\right\|_{\fH} \le\|\tilde\phi\|_\fH+\left\|\int_{\bS\fH}\tilde\phi\rho_f(t,d\tilde\phi)\right\|_{\fH}
\le 1+\int_{\bS\fH}\|\tilde\phi\|\rho_f(t,d\tilde\phi)=2\,.
\]
Finally, we combine \eqref{D-3-4} and \eqref{D-3-5} to derive Gronwall's inequality for $D_N$:
\[
\dot{D}_N(t)\le \Big( 3L_1+\frac{7\kappa}{2}  \Big) D_N(t)+\frac{\kappa}{N} =: L_2 D_N(t) + \frac{\kappa}{N}. 
\]
This yields the desired quantitative estimate for $D_N$:
\begin{equation} \label{D-5}
D_N(t)\le D_N(0)\exp(L_2t)+\frac{\kappa}{N}\frac{e^{L_2t}-1}{L_2}\,,\quad t\ge 0\,.
\end{equation}
\subsection{Using the $N$-particle symmetry} \label{sec:4.4}
In this subsection, we present our first main result. In the next lemma, we show that $W_2(F_{N:m}(t),f(t)^{\otimes m})$ can be controlled by $D_N(t)$. 
\begin{Lem} \label{L4.3}
Suppose that $Q^{0}_N$ is a symmetric coupling of $F^{0}_N$ and of $(f^{0})^{\otimes N}$. Then we have
\[ 
W_2(F_{N:m}(t),f(t)^{\otimes m})^2 \leq 2mD_N(t),\quad t\ge 0\,,
\]
for each $m \in [N]$.
\end{Lem}
\begin{proof}
Since $Q^0_N$ is the symmetric coupling of $F^0_N$ and of $(f^0)^{\otimes N}$, the second assertion in Lemma \ref{L4.1} implies that $Q_N(t)$ is a symmetric coupling 
of $F_N(t)$ and $f(t)^{\otimes N}$ for each $t\ge 0$. Thus, one has
\begin{align*}
\begin{aligned}
& \int_{\Lambda_N\times\Lambda_N}(\|\phi_k-\tilde\phi_k\|_\fH^2+\|V_k-\tilde V_k\|_{L^\infty}^2)Q_N(t,d\Phi_Nd\cV_Nd\tilde\Phi_Nd\cV_N) \\
& \hspace{1cm} =\int_{\Lambda_N\times\Lambda_N}(\|\phi_1-\tilde\phi_1\|_\fH^2+\|V_1-\tilde V_1\|_{L^\infty}^2)Q_N(t,d\Phi_Nd\cV_Nd\tilde\Phi_Nd\cV_N)
\end{aligned}
\end{align*}
for all $k \in [N]$. Therefore, one has
\[
D_N(t)= \frac{1}{2} \int_{\Lambda_N\times\Lambda_N}(\|\phi_k-\tilde\phi_k\|_\fH^2+\|V_k-\tilde V_k\|_{L^\infty}^2)Q_N(t,d\Phi_Nd\cV_Nd\tilde\Phi_Nd\cV_N)
\]
for all $k \in [m]$, so that
\begin{align}
\begin{aligned} \label{D-6}
D_N(t) &=\int_{\Lambda_N\times\Lambda_N} \frac{1}{2m} \sum_{k=1}^m(\|\phi_k-\tilde\phi_k\|_\fH^2+\|V_k-\tilde V_k\|_{L^\infty}^2)Q_N(t,d\Phi_Nd\cV_Nd\tilde\Phi_Nd\cV_N) \\
&=\int_{\Lambda_N\times\Lambda_N}\frac{1}{2m} \sum_{k=1}^m(\|\phi_k-\tilde\phi_k\|_\fH^2+\|V_k-\tilde V_k\|_{L^\infty}^2)Q_{N:m}(t,d\Phi_md\cV_md\tilde\Phi_md\cV_m).
\end{aligned}
\end{align}
Since $Q_{N:m}(t)$ is obviously a coupling of $F_{N:m}(t)$ and of $f(t)^{\otimes m}$, one concludes that
\begin{align}
\begin{aligned} \label{D-6-1}
& \int_{\Lambda_N\times\Lambda_N}\frac1m\sum_{k=1}^m(\|\phi_k-\tilde\phi_k\|_\fH^2+\|V_k-\tilde V_k\|_{L^\infty}^2)Q_{N:m}(t,d\Phi_md\cV_md\tilde\Phi_md\cV_m) \\
& \hspace{2cm} \ge\frac1m W_2(F_{N:m}(t),f(t)^{\otimes m})^2.
\end{aligned}
\end{align}
Finally, we combine \eqref{D-6} and \eqref{D-6-1} to get the desired estimate.
\end{proof}
Now, we are ready to state our first main result.
\begin{Thm} \label{T4.1}  
Let $f$ be the solution to \eqref{A-2} with the initial datum $f^{0}  \in\cP_2( \bS\fH\times C(\bbt^d))$, and let $F_N$ be the solution of the Liouville equation \eqref{C-2} with initial data $F_N^{0}=(f^{0})^{\otimes N}$. 
Then, the first marginal of $F_{N:1}$ of $F_N$ satisfies
\[  
W_2(F_{N:1}(t),f(t))^2   \le  \frac{2\kappa}{N}\frac{e^{L_2t}-1}{L_2}, \qquad\text{ for all }t\ge 0,
\]
where $W_2$ is 2-Wasserstein distance (see Chapter 7 in \cite{Vi} for its definition and its basic properties).
\end{Thm}
\begin{proof}  
It follows from Lemma \ref{L4.3} and \eqref{D-5} that $W_2(F_{N:1}(t),f(t))^2$ and  $D_N$ satisfy
\begin{equation} \label{D-7}
W_2(F_{N:1}(t),f(t))^2  \leq  2D_N(t) \leq 2D_N(0)\exp(L_2t)+\frac{2\kappa}{N}\frac{e^{L_2t}-1}{L_2}\,,
\end{equation}
for all $t\ge 0$. Now, we choose the initial coupling $Q_N^{0}$ of the form 
\[
Q_N^{0}=\prod_{j=1}^Nf^{0}(d\phi_j dV_j)\de_{\phi_j}({\tilde \phi}_j)\de_{V_j}({\tilde V}_j).
\]
Then, by definition of $D_N(0)$, one has 
\begin{equation} \label{D-8}
D_N(0)=0.
\end{equation}
Finally, we combine \eqref{D-7} and \eqref{D-8} to find the desired estimate.
\end{proof}
\begin{Rmk}
Using Lemma \ref{L4.3} and the second inequality in \eqref{D-7}, one finds that, for each $m \in [N]$, and under the same assumptions as in Theorem \ref{T4.1}
\[  
W_2(F_{N:m}(t),f(t)^{\otimes m})^2   \le  \frac{2m\kappa}{N}\frac{e^{L_2t}-1}{L_2}, \qquad\text{ for all }t\ge 0,
\]
\end{Rmk}
\vspace{0.2cm}

In the following two sections, we study a priori emergent dynamics (complete synchronization and practical synchronization) of the kinetic SL equation \eqref{A-2}.


\section{Complete synchronization of the kinetic SL equation}  \label{sec:5}
\setcounter{equation}{0}


In this section, we study the emergent dynamics of the kinetic SL equation \eqref{A-2} in the case of zero Hamiltonians and Hartree term:
\[  
V = 0, \quad \omega \equiv 0,
\]
and we assume a mono-kinetic ansatz for $f$: 
\[ 
f(t, d\phi dV) = \rho(t,d\phi)\delta_0(dV), 
\]
where $\rho$ is a time-dependent Borel probability measure on $\bS\fH$ and $\delta_0$ is the Dirac measure concentrated at $V = 0$. In this setting, the local mass density 
\[ \rho(t, d\phi) =  \int_{C(\bT^d)}  f(t, d\phi dV) \]
satisfies
\begin{equation} 
\begin{cases} \label{E-1}
\displaystyle \d_t \rho + \frac{\kappa}{2} \Div_\phi\left(\rho \int_{\bS\fH} K(\phi, \tilde \phi) \rho(t,d\tilde \phi)\right)=0, \quad (t, \phi) \in \bbr_+ \times \bS\fH, \\
\displaystyle \rho \Big|_{t = 0} = \rho^0,
\end{cases}
\end{equation}
where the kernel $K$ is given as follows:
\begin{equation} \label{E-2}
K(\phi, {\tilde \phi}) = \tilde\phi-\la\phi|\tilde\phi\ra\phi =  (I-\bP_\phi) \tilde \phi.
\end{equation}
Next, we show that the total mass is conserved along \eqref{E-1}.
\begin{Lem} \label{L5.1} Let $\rho= \rho(t,d\phi)$ be a weak solution to \eqref{E-1}. Then, the total mass is conserved along the dynamics \eqref{E-1}.
\[  
\int_{\bS\fH} \rho(t,d\phi) =  \int_{\bS\fH} \rho^0(d\phi), \quad t > 0.
\]
\end{Lem}
\begin{proof} It follows from the weak formulation of \eqref{E-1} in Lemma \ref{L-B.1} that 
\begin{equation} \label{E-3}
\int_{\bS\fH} h(\phi)\rho(t,d\phi) = \int_{\bS\fH} h(\phi)\rho^0(d\phi) +  \frac{\kappa}{2} \int_0^t \int_{(\bS\fH)^2} \langle dh(\phi), K(\phi, \tilde \phi) \rangle \rho(s, d\phi) \rho(t, d\tilde \phi) ds,
\end{equation}
for a test function $h \in C^1(\bS\fH)$. Now, we choose $h \equiv 1$ in \eqref{E-3}  to get the desired conservation of total mass.
\end{proof}
\subsection{Particle path} \label{sec:5.1}
In this subsection, we study the particle path associated with the continuity equation \eqref{E-1} with a nonlocal flux. Let $t \mapsto \rho(t)$ be continuous on $[0, \infty)$ with values in the set $\cP(\bS\fH)$ of Borel probability measures on $\bS\fH$. In what follows, we assume that 
\begin{equation} \label{E-3-1}
\int_{\bS\fH} \rho(t,d\phi) = 1, \quad t \geq 0. 
\end{equation}
Next, we consider the particle path given by the solution for the Cauchy problem:
\begin{equation}  \label{E-4}
\begin{cases}
\displaystyle \frac{d \phi(t)}{dt} =  \frac{\kappa}{2} \int_{\bS\fH} K(\phi(t), {\tilde \phi}) \rho(t, d{\tilde \phi}), \quad t > 0,  \\
\displaystyle \phi \Big|_{t = 0}  = \phi^0 \in\bS\fH.
\end{cases}
\end{equation}
Note that the unique solvability of $\phi = \phi(t)$ is guaranteed by the Cauchy-Lipschitz theory in Appendix \ref{S-CL} (see Theorem \ref{T-MFFlow}). Let ${\mathcal S}(t)$ be the associated flow map. Then, we have the following relation (see \cite{A-G, G-H} for more details):
\begin{equation} \label{E-4-1}
\rho(t) = {\mathcal S}(t) \# \rho^0\,, \quad\text{ and hence} \quad \mbox{supp}(\rho(t)) = {\mathcal S}(t)(\mbox{supp}\,\rho^0)\,.
\end{equation}
\subsection{Asymptotic dynamics} \label{sec:5.2}
For the emergent dynamics of \eqref{E-1}, we recall the functional corresponding to the diameter of the $\phi$-support of $\rho(t, \cdot)$:
\begin{equation} \label{E-5}
{\mathcal D}[\rho(t)] = \sup \Big \{ \| \phi - {\tilde \phi} \|_{\fH}:~\forall~\phi,~{\tilde \phi} \in \mbox{supp}\,\rho(t) \Big \}. 
\end{equation}
Then, it is easy to see that 
\[
{\mathcal D}[\rho(t)] \leq 2,
\]
We now recall the concept of complete synchronization as follows.
\begin{Def} \label{D5.1}
Let $\rho = \rho(t, \cdot)$ be a global Lagrangian weak solution to \eqref{E-1}  given by the relation \eqref{E-4-1} (see Remark \ref{RB.1}). Then, $\rho$ exhibits ``{\it complete (state) synchronization}'' if $\rho$ satisfies
\[  
\lim_{t \to \infty} {\mathcal D}[\rho(t)] = 0. 
\]
\end{Def}
In the sequel, we study the temporal evolution of ${\mathcal D}[\rho(t)]$. Let $\phi_1(t) = \phi_1(t; \phi^{0}_{1})$ and $\phi_2(t) = \phi_2(t; \phi^{0}_{2})$ be two global solutions to \eqref{E-4} corresponding to the initial data $\phi^{0}_{1}$ and $\phi^{0}_{2}$ at time $t = 0$, respectively. 
Then, it follows from  $\eqref{E-4}$ that for $t \geq 0$,
\begin{equation} \label{E-6}
{\dot \phi}_1(t) =  \frac{\kappa}{2} \int_{\bS\fH}  K(\phi_1(t), {\tilde \phi}) \rho(t,d{\tilde \phi}), \quad 
{\dot \phi}_2(t) =  \frac{\kappa}{2} \int_{\bS\fH}  K(\phi_2(t), {\tilde \phi}) \rho(t,d{\tilde \phi}).
\end{equation}
Next, we study the temporal evolution of $\|\phi_1(t) - \phi_2(t) \|_{\fH}^2$.
\begin{Lem} \label{L5.2}
Let $\phi_1 = \phi_1(t)$ and $\phi_2 = \phi_2(t)$ be two global solutions to \eqref{E-4} corresponding to the initial data $\phi_1^0,~\phi_2^0 \in \mbox{supp}\,\rho^0$, respectively. Then, one has
\begin{equation} \label{E-7}
\frac{d}{dt}  \|\phi_1(t) - \phi_2(t) \|_{\fH} \leq  \kappa \Big(  -\frac{1}{2} +  {\mathcal D}[\rho(t)] \Big)\| \phi_1(t) - \phi_2(t) \|_{\fH}, \quad t > 0. 
\end{equation}
\end{Lem}
\begin{proof} 
It follows from \eqref{E-2} and \eqref{E-6} that 
\begin{align}
\begin{aligned} \label{E-8}
\frac{d}{dt} (\phi_1(t) - \phi_2(t)) &= \frac{\kappa}{2} \int_{\bS\fH} \Big( K(\phi_1(t), {\tilde \phi}) -  K(\phi_2(t), {\tilde \phi}) \Big) \rho(t,d{\tilde \phi})  \\
&=  \frac{\kappa}{2} \int_{\bS\fH}  \Big( \la\phi_2(t)|\tilde\phi\ra\phi_2(t) -  \la\phi_1(t) |\tilde\phi\ra\phi_1(t)  \Big) \rho(t,d{\tilde \phi}),
\end{aligned}
\end{align}
where we use the relation:
\[ K(\phi_1, {\tilde \phi}) -  K(\phi_2, {\tilde \phi}) = \la\phi_2|\tilde\phi\ra\phi_2 -  \la\phi_1 |\tilde\phi\ra\phi_1. \]
Now, we use \eqref{E-8} to find 
\begin{align}
\begin{aligned} \label{E-9}
 &\frac{d}{dt}  \|\phi_1(t) - \phi_2(t) \|_{\fH}^2 = \frac{d}{dt} \langle \phi_1(t) - \phi_2(t)| \phi_1(t) - \phi_2(t) \rangle \\
 &= \Big \langle \frac{d}{dt} ( \phi_1(t) - \phi_2(t)) \Big|  \phi_1(t) - \phi_2(t) \Big \rangle +  \Big \langle \phi_1(t) - \phi_2(t) \Big | \frac{d}{dt} (\phi_1(t) - \phi_2(t)) \Big \rangle \\
 &=  \frac{\kappa}{2} \Big \langle  \int_{\bS\fH} \Big( \la\phi_2(t) |\tilde\phi\ra\phi_2(t) -  \la\phi_1(t) |\tilde\phi\ra\phi_1(t) \Big)  \rho(t,d{\tilde \phi}) \Big|  \phi_1(t) - \phi_2(t) \Big \rangle \\
 &\hspace{0.2cm}+  \frac{\kappa}{2} \Big \langle \phi_1(t) - \phi_2(t) \Big|  \int_{\bS\fH} \Big( \la\phi_2(t) |\tilde\phi\ra\phi_2(t) -  \la\phi_1(t) |\tilde\phi\ra\phi_1(t) \Big)  \rho(t, d{\tilde \phi}) \Big \rangle \\
 &= \frac{\kappa}{2}  \int_{\bS\fH}  \la\phi_2(t) |\tilde\phi\ra \langle \phi_2(t) | \phi_1(t)\!-\!\phi_2(t) \rangle \rho(t,d{\tilde \phi})\!-\!  \frac{\kappa}{2}   \int_{\bS\fH}  \la\phi_1(t) |\tilde\phi\ra  \langle \phi_1(t)|   \phi_1(t)\!-\! \phi_2(t) \rangle \rho(t,d{\tilde \phi})  \\
 & \hspace{0.2cm} +  \frac{\kappa}{2}  \overline{\int_{\bS\fH} \la\phi_2(t) |\tilde\phi \ra \langle \phi_2(t), \phi_1(t)\!-\!\phi_2(t) \rangle \rho(t,d{\tilde \phi}) }\!-\! \frac{\kappa}{2} \overline{ \int_{\bS\fH} \la\phi_1(t) |\tilde\phi \ra  \langle  \phi_1(t), \phi_1(t)\!-\! \phi_2(t) \rangle \rho(t,d{\tilde \phi}) }    \\
 &= \kappa \mbox{Re} \Big[  \int_{\bS\fH} \underbrace{\Big( \la\phi_2(t)|\tilde\phi\ra  \langle \phi_2(t)| \phi_1(t) - \phi_2(t) \rangle -  \la\phi_1(t) |\tilde\phi\ra \langle \phi_1(t)|   \phi_1(t) - \phi_2(r) \rangle \Big)}_{=: {\mathcal I}_3(t)}\rho(t,d{\tilde \phi})      \Big].
\end{aligned}
\end{align}
Next, we split ${\mathcal I}_3$ into several pieces:
\begin{align}
\begin{aligned} \label{E-10}
{\mathcal I}_3 &= \la\phi_2|\tilde\phi\ra \langle \phi_2| \phi_1 - \phi_2 \rangle  - \la\phi_1 |\tilde\phi\ra \langle \phi_1|   \phi_1 - \phi_2 \rangle  \\
& = \la\phi_2|\tilde\phi\ra \langle \phi_2 - \phi_1| \phi_1 - \phi_2 \rangle + \la\phi_2|\tilde\phi\ra \langle \phi_1| \phi_1 - \phi_2 \rangle - \la\phi_1 |\tilde\phi\ra \langle \phi_1|   \phi_1 - \phi_2 \rangle \\
& =  \la\phi_2|\tilde\phi\ra \langle \phi_2 - \phi_1| \phi_1 - \phi_2 \rangle + \la\phi_2 - \phi_1|\tilde\phi\ra \langle \phi_1| \phi_1 - \phi_2 \rangle  \\
&=  \la\phi_2| \phi_2 \ra \langle \phi_2 - \phi_1| \phi_1 - \phi_2 \rangle + \la\phi_2|\tilde\phi - \phi_2 \ra \langle \phi_2 - \phi_1| \phi_1 - \phi_2 \rangle \\
& \hspace{0.2cm}  +  \la\phi_2 - \phi_1|\tilde\phi\ra \langle \tilde{\phi}| \phi_1 - \phi_2 \rangle  + \la\phi_2 - \phi_1|\tilde\phi\ra \langle \phi_1 - {\tilde \phi}| \phi_1 - \phi_2 \rangle    \\
& =: {\mathcal I}_{31} +  {\mathcal I}_{32} +  {\mathcal I}_{33} +  {\mathcal I}_{34}.
\end{aligned}
\end{align}
Below, we estimate the terms ${\mathcal I}_{3i},~i = 1,2,3,4$ one by one. \newline

\noindent $\bullet$~Case A (Estimate of ${\mathcal I}_{31}$ and ${\mathcal I}_{33}$):~We use $\|\phi_2 \|_{\fH} = 1$ to find 
\begin{equation} \label{E-11}
{\mathcal I}_{31} = -  \| \phi_1 - \phi_2 \|_{\fH}^2, \quad {\mathcal I}_{33} =   -|\la\phi_2 - \phi_1|\tilde\phi\ra|^2.    
\end{equation}
\noindent $\bullet$~Case B (Estimate of ${\mathcal I}_{32}$):~We use $\| \phi_2 \|_{\fH} = 1$ and the Cauchy-Schwartz inequality to find 
\begin{equation} \label{E-12}
|{\mathcal I}_{32}| \leq  \| \phi_2 \|_{\fH} \cdot \| \tilde\phi - \phi_2 \|_{\fH} \cdot \| \phi_2 - \phi_1 \|^2_{\fH}  = \| \tilde\phi - \phi_2 \|_{\fH} \|\phi_1 - \phi_2 \|_{\fH}^2.
\end{equation}
\noindent $\bullet$~Case C (Estimate of ${\mathcal I}_{34}$):~Similar to Case B, we have
\begin{equation} \label{E-13}
|{\mathcal I}_{34}| \leq  |\la\phi_2 - \phi_1|\tilde\phi\ra| \cdot |\langle \phi_1 - {\tilde \phi}| \phi_1 - \phi_2 \rangle | \leq \|  \phi_1 - {\tilde \phi} \|_{\fH}  \|  \phi_1 - \phi_2 \|_{\fH}^2.
\end{equation}
In \eqref{E-9}, we combine all the estimates \eqref{E-10}, \eqref{E-11}, \eqref{E-12} and \eqref{E-13} to find 
\begin{align*}
\begin{aligned}
& \frac{d}{dt}  \|\phi_1(t) - \phi_2(t) \|_{\fH}^2  \\
& \hspace{0.5cm} \leq   \kappa \Big[  -1 +  \int_{\bS\fH} \Big( \| \tilde\phi - \phi_2(t) \|_{\fH}  +   \|  \phi_1(t) - {\tilde \phi} \|_{\fH}  \Big )\rho(t,d{\tilde \phi}) \Big ]  \| \phi_1(t) - \phi_2(t) \|_{\fH}^2 \\
& \hspace{0.7cm} - \underbrace{\int_{\bS\fH} |\la\phi_2 - \phi_1|\tilde\phi\ra|^2 \rho(t,d{\tilde \phi})}_{\geq 0} \\
& \hspace{0.5cm} \leq  \kappa \Big[  -1 +  \int_{\bS\fH} \Big( \| \tilde\phi - \phi_2(t) \|_{\fH}  +   \|  \phi_1(t) - {\tilde \phi} \|_{\fH}  \Big )\rho(t,d{\tilde \phi}) \Big ]  \| \phi_1(t) - \phi_2(t) \|_{\fH}^2 \\
& \hspace{0.5cm}= \kappa \Big[  -1 +  \int_{\mbox{supp}(\rho)} \Big( \| \tilde\phi - \phi_2(t) \|_{\fH}  +   \|  \phi_1(t) - {\tilde \phi} \|_{\fH}  \Big )\rho(t,d{\tilde \phi}) \Big ]  \| \phi_1(t) - \phi_2(t) \|_{\fH}^2 \\
& \hspace{0.5cm} \leq \kappa \Big(  -1 +  2{\mathcal D}[\rho(t)] \Big)\| \phi_1(t) - \phi_2(t) \|_{\fH}^2,
\end{aligned}
\end{align*}
where we used the unit mass condition \eqref{E-3-1} and \eqref{E-5} to see
\[
 \int_{\mbox{supp}(\rho(t))} \Big( \| \tilde\phi - \phi_2(t) \|_{\fH}  +   \|  \phi_1(t) - {\tilde \phi} \|_{\fH}  \Big )\rho(t,d{\tilde \phi}) \leq  2{\mathcal D}[\rho(t)]  \int_{\mbox{supp}(\rho(t))} \rho(t,d{\tilde \phi}) 
 \leq  2{\mathcal D}[\rho(t)].
\]
Therefore, we have
\[ \frac{d}{dt}  \|\phi_1(t) - \phi_2(t) \|_{\fH}^2 \leq  \kappa \Big(  -1 +  2{\mathcal D}[\rho(t)] \Big)\| \phi_1(t) - \phi_2(t) \|_{\fH}^2. \]
This yields the desired estimate \eqref{E-7}. 
\end{proof}
\noindent Now, we are ready to prove the exponential decay of ${\mathcal D}[\rho(t)]$. For $\eta\in \Big [0, \frac{1}{16} \Big )$, we consider the equation motivated by the right-hand side of \eqref{E-7}:
\[
z^2-  \frac{1}{2} z+\eta=0
\]
Then, it has two nonnegative roots $\zeta_1(\eta)$ and $\zeta_2(\eta)$ such that
\[
0\le\zeta_1(\eta)< \frac{1}{4} <\zeta_2(\eta)\le 2.
\]
For all $z\ge 0$, one has
\[
z^2- \frac{1}{2} z+\eta<0 \quad \hbox{ if and only if } \quad \zeta_1(\eta)<z<\zeta_2(\eta).
\]
Note that
\[
\zeta_1(0)=0\quad\hbox{ and }\quad\zeta_2(0)= \frac{1}{2}.
\]

\begin{Lem}\label{L5.3}
For a subset $\Si \subset[0, \frac{1}{2})$, we assume that coupling strength and initial data satisfy 
\[ 
\kappa > 0, \quad 0 < \Dlt^0 := \sup(\Si) < \frac{1}{2}, 
\]
and let $x\in C^1([0,+\infty)$ be a function satisfying the following relations:
\[
\begin{cases}
\displaystyle 0 \le x(t)\le 2, \quad t \geq 0, \\
\displaystyle \dot x(t)\le  \kappa \Big(- \frac{1}{2} +  \Dlt(t) \Big) x(t),  \quad t > 0, \\
\displaystyle  x \Big|_{t = 0} =x^0, \quad \Dlt(t):=\sup_{x^0\in\Si}x(t), \quad  t \geq 0.
\end{cases}
\]
Then we have
\[
\Dlt(t)\le y(t)\quad\hbox{ for each }t\ge 0,
\]
where $y$ is the solution of the Cauchy problem for the Riccati equation:
\[
\begin{cases}
\displaystyle \dot y= -\frac{\ka}{2} y+  \ka y^2,\quad t > 0, \\
\displaystyle y \Big|_{t = 0} =\Dlt^0.
\end{cases}
\]
\end{Lem}
\begin{proof}
We refer to the same arguments as in Appendix B in \cite{G-H} for a proof. 
\end{proof}

\begin{Thm} \label{T5.1}
Suppose that coupling strength and initial datum satisfy
\begin{equation*} \label{E-11-1}
\kappa > 0, \quad  \rho^0\in\cP(\bS\fH), \quad {\mathcal D}[\rho^{0}] < \frac{1}{2}, \quad 
\end{equation*}  
and let $\rho$ be a Lagrangian weak solution to \eqref{E-1}. Then ${\mathcal D}[\rho]$ decays to zero exponentially fast as $t \to \infty$:
\[    
{\mathcal D}[\rho(t)]  \leq  \frac{\mathcal D[\rho^0]}{(1-2\mathcal D[\rho^0])e^{\frac{\kappa}{2} t}+2\mathcal D[\rho^0]}, \quad t \geq 0. 
\]
\end{Thm}

\begin{proof} It follows from \eqref{E-7} that 
\[ \frac{d}{dt}  \|\phi_1(t) - \phi_2(t) \|_{\fH} \leq  \kappa \Big(  -\frac{1}{2} + {\mathcal D}[\rho(t)] \Big)\| \phi_1(t) - \phi_2(t) \|_{\fH}, \quad t > 0. \]
In Lemma \ref{L5.3}, we set
\[ 
x(t) =  \| \phi_1(t) - \phi_2(t) \|_{\fH}, \quad  {\mathcal D}[\rho(t)] = \Delta(t). 
\]
Then, we have
\begin{equation} \label{E-13-1}
{\mathcal D}[\rho(t)] \leq y(t), 
\end{equation}
where $y = y(t)$ is a solution of the following Riccati equation:
\begin{equation} \label{E-14}
\begin{cases}
\displaystyle \dot y= -\frac{\ka}{2} y+  \ka y^2,\quad t > 0,\\
\displaystyle y \Big|_{t = 0} ={\mathcal D}[\rho^0].
\end{cases}
\end{equation}
Then, it is easy to see that $y = y(t)$ is explicitly given by 
\begin{equation} \label{E-15}
y(t) = \frac{{\mathcal D}[\rho^0]}{2 {\mathcal D}[\rho^0] + (1-2 {\mathcal D}[\rho^0]) e^{\frac{\kappa}{2} t}}.
\end{equation}
Finally, we combine \eqref{E-13-1} and \eqref{E-15} to get the desired estimate. 
\end{proof}


\section{Practical synchronization of the kinetic SL equation} \label{sec:6}
\setcounter{equation}{0}

In this section, we study the emergent dynamics for the kinetic SL equation in the case of non-identical Hamiltonians and zero Hartree potential with $\omega \equiv 0.$
Compared to the previous section, we may not be able to derive the complete synchronization, but we instead obtain a weaker concept of synchronization, namely "{\it practical synchronization}" which can control the size of asymptotic $\phi$-diameter ${\mathcal D}[f_\kappa] := {\mathcal D}[\rho_\kappa]$ by increasing the coupling strength as we want (see Definition \ref{D6.1} in Section \ref{sec:6.2}).

\subsection{Particle path and diameter functional} \label{sec:6.1}

Recall the kinetic SL equation for the time-dependent density function $f = f(t, \cdot, \cdot)$ on $\bS\fH \times C(\bbt^d)$:
\begin{equation} \label{F-0}
\d_tf+\Div_\phi (f Y^f) +\frac{\kappa}{2} \Div_\phi\left(f \int_{\bS\fH \times C(\bbt^d)} K(\phi, \tilde\phi) f(t, d{\tilde \phi}d{\tilde V})\right)=0.
\end{equation}
Following the same argument as in Section \ref{sec:5.1}, we consider a particle path (forward characteristics) associated with \eqref{F-0}:~for $(\phi^0, V^0) \in \bS\fH \times C(\bbt^d)$, the particle path 
\[ (\phi(t|\phi^0), V(t|V^0)) := (\phi(t| 0, \phi^0, V^0), V(t| 0, \phi^0, V^0)) \]
is defined as the unique solution for the following Cauchy problem:
\begin{equation}\label{F-1}
\begin{cases}
\displaystyle \partial_t\phi(t|\phi^{0})=Y^f(t,\phi(t|\phi^{0}),V(t|V^0))+\tfrac\kappa2Y^{mf}[\rho_f](t,\phi(t|\phi^{0}))\,, \quad t > 0, \\
\displaystyle \partial_t V(t| V^0) = 0, \\
\displaystyle \phi(0|\phi^{0})=\phi^{0}, \quad V(0|V^{0})= V^{0},
\end{cases}
\end{equation}
where the vector fields $Y^f$ and $Y^{mf}$ are introduced in \eqref{DiffEq-phi} and \eqref{New-1-1}, respectively. The unique solvability of \eqref{F-1} is guaranteed by Theorem \ref{T-MFFlow} and it is easy to see that 
\[ 
V(t|V^0) = V^0, \quad f = (\phi(t|\cdot), V(t|V^0)) \# f^0.
\]
Now, we denote 
\[ \pi_1:\,\bS\fH\times C(\bbt^d)\ni(\phi,V)\mapsto\phi\in\bS\fH \quad \mbox{and} \quad \rho_f(t):=\pi_1\#f(t). \]
In other words
\[ 
\rho_f(t, \phi):= \int_{C(\bbt^d)} f(t, d\phi dV),
 \]
together with
\begin{equation} \label{DefD(t)}
\mathcal D[\rho(t)] := \sup \Big \{ \| \phi_1 - \phi_2 \|_{\fH}~\Big|~\phi_1, \phi_2 \in \mbox{supp} (\rho_f(t, \cdot)) \Big \}, 
\end{equation}
and 
\begin{equation} \label{Defa}
\alpha := \tfrac1\hbar\sup \Big \{ \|V_1 - V_2 \|_{L^{\infty}}~|~(\phi_1, V_1)~\mbox{and}~(\phi_2, V_2) \in \mbox{supp}(f(t, \cdot)) \Big \}. 
\end{equation}
Since the $V$-component of the vector field associated with \eqref{F-0} is identically zero, the quantity $\alpha$ is independent of $t$.


\subsection{Practical synchronization} \label{sec:6.2}

In this subsection, we provide the practical synchronization estimate for the kinetic SL equation \eqref{F-0}. Assume that the initial datum $f^0$ has a compact $V$-support:
\[ P_\infty :=  \sup \Big \{ \|V \|_{L^{\infty}}~\Big |~(\phi, V)  \in \mbox{supp}(f^0(\cdot)) \Big \} < \infty.  \]
Since $V$-variable in $f$ is stationary, we also have
\begin{equation} \label{F-2}
\sup \Big \{ \|V \|_{L^{\infty}}~\Big |~(\phi, V)  \in \mbox{supp}(f(t,\cdot)) \Big \} =  P_\infty  < \infty.  
\end{equation}
We choose $(\phi_1^0, V_1)$ and $(\phi_2^0, V_2)$ in $\mbox{supp}(f^0)$, and we set
\[ 
\phi_1(t) = S[V_1](t) \phi_1^0, \quad   \phi_2(t) = S[V_2](t) \phi_2^0,
\]
where $S[V](t)$ is the solution operator generated by \eqref{F-1} with $V^0 = V$, i.e., $\phi_1$ and $\phi_2$ satisfy
\begin{equation} \label{F-5}
\begin{cases}
\displaystyle  {\dot \phi}_1(t) = Y^f(t, \phi_1(t), V_1) + \frac{\kappa}{2} \int_{\bS\fH} K(\phi_1(t), \tilde \phi) \rho_f(t, d\tilde \phi), 
\\
\displaystyle  {\dot \phi}_2(t) = Y^f(t, \phi_1(t), V_2) + \frac{\kappa}{2} \int_{\bS\fH} K(\phi_2(t), \tilde \phi) \rho_f(t, d\tilde \phi).
\end{cases}
\end{equation}
Next, we study the time-rate of change for $\| \phi_1(t) - \phi_2(t) \|_{\fH}$. For this, we use $\eqref{F-5}$ to get 
\begin{align*}
\begin{aligned}
\frac{d}{dt} (\phi_1(t) - \phi_2(t)) &= \Big( Y^f(t, \phi_1(t), V_{1}) - Y^f(t, \phi_2(t), V_{2})  \Big)
\\
&+  \frac{\kappa}{2} \int_{\bS\fH} \Big( K(\phi_1(t), {\tilde \phi})  - K(\phi_2(t), {\tilde \phi})  \Big) \rho_f(t, d{\tilde \phi})
\\
&=: {\mathcal I}_{41}(t) + {\mathcal I}_{42}(t).
\end{aligned}
\end{align*} 
This yields
\begin{align}
\begin{aligned} \label{F-7}
\frac{d}{dt} \| \phi_1(t) - \phi_2(t) \|_{\fH}^2 &=  \Big \langle  {\mathcal I}_{41}(t) \Big | \phi_1(t) - \phi_2(t) \Big \rangle +  \Big \langle \phi_1(t) - \phi_2(t)  \Big |   {\mathcal I}_{41}(t) \Big \rangle \\
&+ \Big \langle  {\mathcal I}_{42}(t) \Big | \phi_1(t) - \phi_2(t)  \Big \rangle + \Big \langle \phi_1(t) - \phi_2(t) \Big |   {\mathcal I}_{42}(t)  \Big \rangle.
\end{aligned}
\end{align}
In the following lemma, we provide estimates for the right-hand side of \eqref{F-7}.
\begin{Lem} \label{L6.1}
Let $(\phi_1(t))$ and $(\phi_2(t))$ be the projected particle paths satisfying \eqref{F-5}. Then, one has the following estimates:~for $t \geq 0$, 
\begin{eqnarray*}
&& (i)~ \Big| \Big \langle  {\mathcal I}_{41} (t) \Big | \phi_1(t) - \phi_2(t) \Big \rangle + \Big \langle \phi_1(t) - \phi_2(t) \Big |   {\mathcal I}_{41}(t) \Big \rangle \Big| 
\\
&& \hspace{1cm} \leq \tfrac2\hbar \|V_{1}\|_{L^\infty} \|\phi_1(t)- \phi_2(t) \|^2_\fH+\tfrac2\hbar\|V_{1}-V_{2} \|_{L^\infty}  \|\phi_1(t)- \phi_2(t) \|_\fH.   
\cr
&& (ii)~ \Big \langle  {\mathcal I}_{42}(t)  \Big | \phi_1(t) - \phi_2(t) \Big \rangle + \Big \langle \phi_1(t) - \phi_2(t) \Big |   {\mathcal I}_{42}(t)  \Big \rangle 
\cr
&& \hspace{1cm} =  \kappa \Big(  -1 +  2{\mathcal D}[\rho_f(t)] \Big)\| \phi_1(t) - \phi_2(t) \|_{\fH}^2. 
\end{eqnarray*}
\end{Lem}
\begin{proof} 
(i)~First, we note that 
\begin{align}
\begin{aligned} \label{NF-7}
{\mathcal I}_{41}(t) &=  Y^f(t,\phi_1(t),V_1)-Y^f(t, \phi_2(t), V_2) \\
&= \tfrac1{{\mathrm i}\hbar}U(\hbar t)^* V_1 U(\hbar t)(\phi_1(t)- \phi_2(t))
+ \tfrac1{{\mathrm i}\hbar}U(\hbar t)^*(V_1- V_2)U(\hbar t) \phi_2(t).
\end{aligned}
\end{align}
Then, \eqref{NF-7}, $\| \phi_2(t) \|_{\fH} = 1$ and the Cauchy-Schwarz inequality imply
\begin{equation} \label{F-7-0}
\| {\mathcal I}_{41}(t) \|_\fH \le
\tfrac1{\hbar} \|V_1\|_{L^\infty} \|\phi_1(t)- \phi_2(t) \|_\fH
+ \tfrac1\hbar\|V_1- V_2\|_{L^\infty}.
\end{equation}
Now, we use the Cauchy-Schwartz inequality and \eqref{F-7-0} to find the desired estimate:
\begin{align*}
\begin{aligned}
& \Big| \Big \langle  {\mathcal I}_{41}(t)  \Big | \phi_1(t) - \phi_2(t) \Big \rangle + \Big \langle \phi_1(t) - \phi_2(t) \Big |   {\mathcal I}_{41}(t) \Big \rangle \Big|  \\
& \hspace{1cm}  \leq 2 \| {\mathcal I}_{41}(t) \|_{\fH} \| \phi_1(t) - \phi_2(t)  \|_{\fH} \\
& \hspace{1cm} \leq \tfrac2\hbar \|V_{1}\|_{L^\infty} \|\phi_1(t)- \phi_2(t)\|^2_\fH+\tfrac2\hbar\|V_{1}-V_{2}\|_{L^\infty} \|\phi_1(t)- \phi_2(t)\|_\fH. 
\end{aligned}
\end{align*}
\noindent (ii) We use the same arguments as in Lemma \ref{L5.2} to see that 
\begin{align*}
\begin{aligned}
& \Big \langle  {\mathcal I}_{42}(t) \Big | \phi_1(t) - \phi_2(t) \Big \rangle + \Big \langle \phi_1(t) - \phi_2(t) \Big |   {\mathcal I}_{42}(t)  \Big \rangle  \\
& \hspace{1cm} = \kappa \Big[  -1 +  \int_{\mbox{supp} \rho_f(t)} \Big( \| \tilde\phi - \phi_2(t) \|_{\fH}  +   \|  \phi_1(t) - {\tilde \phi} \|_{\fH}  \Big )\rho_f(t,d{\tilde \phi}) \Big ]  \| \phi_1(t) - \phi_2(t) \|_{\fH}^2 \\
&\hspace{1cm} \leq \kappa \Big(  -1 +  2{\mathcal D}[\rho_f(t)] \Big)\| \phi_1(t) - \phi_2(t) \|_{\fH}^2.
\end{aligned}
\end{align*}
\end{proof}
Now, we combine \eqref{F-7} and Lemma \ref{L6.1} to find 
\begin{align}
\begin{aligned} \label{F-7-1}
\frac{d}{dt} \| \phi_1(t) - \phi_2(t) \|_{\fH}^2 &\leq   \tfrac2\hbar \|V_{1}\|_{L^\infty} \|\phi_1(t)- \phi_2(t) \|^2_\fH+\tfrac2\hbar\|V_{1}-V_{2} \|_{L^\infty} \|\phi_1(t)- \phi_2(t) \|_\fH \\
&+  \kappa \Big(  -1 +  2{\mathcal D}[\rho_f(t)] \Big)\| \phi_1(t) - \phi_2(t) \|_{\fH}^2.
\end{aligned}
\end{align}
Then, we use \eqref{F-2} and \eqref{F-7-1}  to find 
\begin{align}
\begin{aligned} \label{F-7-2}
&\frac{d}{dt} \| \phi_1(t) - \phi_2(t) \|_{\fH} \\
& \hspace{0.5cm} \leq \tfrac1\hbar\|V_{1}-V_{2} \|_{L^\infty} + \Big[   \tfrac1\hbar \|V_{1}\|_{L^\infty} - \frac{\kappa}{2}   +   \kappa {\mathcal D}[\rho_f(t)] \Big ]  \| \phi_1(t) - \phi_2(t) \|_{\fH} \\
& \hspace{0.5cm} \leq \alpha + \Big[   \tfrac1\hbar P_\infty - \frac{\kappa}{2}   +   \kappa {\mathcal D}[\rho_f(t)] \Big ]  \| \phi_1(t) - \phi_2(t) \|_{\fH}.
\end{aligned}
\end{align}
We choose $\kappa$ sufficiently large such that 
\begin{equation} \label{F-7-3}
\tfrac1 \hbar P_\infty  -\frac{\kappa}{2}   < -\frac{\kappa}{4}, \quad \mbox{i.e.,} \quad   \kappa > \frac{4 P_\infty}{\hbar}.     
\end{equation}
Then, for such $\kappa$, we use \eqref{Defa}, \eqref{F-7-2} and \eqref{F-7-3} to find 
\begin{equation} \label{F-8}
\frac{d}{dt} \| \phi_1(t) - \phi_2(t) \|_{\fH} \leq  \alpha+ \kappa  \Big(  -\frac{1}{4}  +  {\mathcal D}[\rho_f(t)] \Big)\| \phi_1(t) - \phi_2(t) \|_{\fH}.
\end{equation}
Before we proceed further, we recall the concept of practical synchronization for solutions of the kinetic SL equation \eqref{F-0}, already introduced in Definition \ref{D2.1} in the case of the finite system.
\begin{Def}\label{D6.1}
Let $\{f_\kappa \}$ be a family of solutions to \eqref{F-0} indexed by the coupling strength $\kappa$, and let $\cD[f_\kappa]$ be the sequence of diameters of $\rho_{f_\kappa}$ defined in \eqref{DefD(t)}. This family of solutions exhibits ``practical synchronization''
if the following relation holds:
\[
\lim_{\ka\to \infty}\varlimsup_{t\to \infty} {\mathcal D}[f_\kappa(t)]=0.
\]
\end{Def}
\vspace{0.2cm}

\noindent In order to get an estimate for the diameter ${\mathcal D}_\kappa[t]$, we use the same argument employed in \cite{G-H}. More precisely, for $\eta\in[0, \frac{1}{64})$, we consider the equation motivated by the right-hand side of \eqref{F-8}:
\begin{equation} \label{NNN-1}
 z^2- \frac{1}{4}z+\eta=0
\end{equation}
Then, it has two nonnegative roots $\xi_1(\eta)$ and $\xi_2(\eta)$ such that
\begin{equation}\label{DefXi2}
0\le\xi_1(\eta)< \frac{1}{8} <\xi_2(\eta)\le 2.
\end{equation}
Therefore, for all $z\ge 0$, one has
\[
z^2-\frac{1}{4} z+\eta<0 \quad \hbox{ if and only if } \quad \xi_1(\eta)<z<\xi_2(\eta).
\]
Note that $\xi_1$ and $\xi_2$ are of class $C^1$ on $[0,\tfrac{1}{64})$ by the implicit functions theorem, and that
\begin{equation} \label{F-10}
\xi_1(0)=0\quad\hbox{ and }\quad\xi_2(0)= \frac{1}{4}.
\end{equation}
\begin{Lem}\label{L6.2}
For a subset $\Si \subset[0, \xi_2(\a/(\ka)))$, we assume that coupling strength and initial data satisfy 
\[ 
0 < \alpha < \infty, \quad \ka>  64 \a > 0, \quad 0<\Dlt^0<\xi_2(\a/(\ka)), 
\]
and let $x\in C^1([0,+\infty)$ be a function satisfying the following relations:
\[
\begin{cases}
\displaystyle 0 \le x(t)\le 2, \quad t \geq 0, \\
\displaystyle \dot x(t)\le \a - \frac{\kappa}{4} x(t) + \ka \Dlt(t) x(t),  \quad t > 0, \\
\displaystyle  x \Big|_{t = 0} =x^0, \quad \Dlt(t):=\sup_{x^0\in\Si}x(t), \quad  t \geq 0.
\end{cases}
\]
Then we have
\[
\Dlt(t)\le y(t)\quad\hbox{ for each }t\ge 0,
\]
where $y$ is the solution of the Cauchy problem for the Riccati equation:
\[
\begin{cases}
\displaystyle \dot y= \alpha -\frac{\ka}{4} y+  \ka y^2,\quad t > 0, \\
\displaystyle y \Big|_{t = 0} =\Dlt^0.
\end{cases}
\]
\end{Lem}
\begin{proof}
We refer to the same arguments as in Appendix B in \cite{G-H} for a proof. 
\end{proof}
Now, we are ready to present the practical synchronization estimate.
\begin{Thm} \label{T6.1}
Suppose that the coupling strength, one-body potentials and initial data satisfy a set of conditions:
\[  P_\infty < \infty, \quad \alpha  < \infty, \quad   {\mathcal D}[f_\kappa^0] <\xi_2(\a/(\ka)), \quad \int_{\bS\fH} \rho_k^0(d\phi) = 1, 
\]
and let $\{f_\kappa \}$ be a family of solutions to \eqref{F-0}.  Then, one has the practical synchronization:
\[
\lim_{\ka\to \infty}\varlimsup_{t\to \infty} {\mathcal D}[f_\kappa(t)]=0.
\]
\end{Thm}

\begin{proof} Since we deal with the case $\kappa \gg 1$, we can assume that 
\[ 
 \kappa > \max \Big \{ \frac{4P_\infty}{\hbar},~64 \alpha  \Big \}.
\]
We apply Lemma \ref{L6.2}  for \eqref{E-8} with the following set-up:
\[
 x(t) := \|\phi_1- \phi_2\|_\fH, \quad  \Delta(t) := {\mathcal D}[f_\kappa].
\]
to get 
\begin{equation} \label{F-11}
{\mathcal D}[f_\kappa(t)] \le y_\kappa(t) \quad\hbox{ for each }t\ge 0, 
\end{equation}
where $y_\kappa$ is the solution of the Cauchy problem:
\[
\begin{cases}
\displaystyle \dot y= \alpha -\frac{\ka}{4} y+  \ka y^2,\quad t > 0, \\
\displaystyle y \Big|_{t = 0} =  {\mathcal D}[f_\kappa^0].
\end{cases}
\]
Now, we combine \eqref{F-10} and \eqref{F-11} to get the desired practical synchronization:
\[ 
\lim_{\ka\to \infty}\varlimsup_{t\to \infty} {\mathcal D}[f_\kappa(t)] \leq  \lim_{\ka\to+\infty} \xi_1 \Big(\frac{\alpha}{\kappa} \Big) = 0.
\]
\end{proof}


\section{Conclusion} \label{sec:7}
\setcounter{equation}{0}


In this paper, we have derived the kinetic SL equation which turns out to be a Vlasov-McKean type equation on the phase space $\bS\fH \times C(\bbt^d)$, and we provided emergent synchronization estimates for the derived kinetic equation using the nonlinear functional approach measuring the diameter for the support of one-oscillator distribution. The SL model was originally introduced by M. Lohe as a non-abelian infinite-dimensional generalization of the Kuramoto model for synchronization and it can be easily reduced to several synchronization models, e.g., swarm sphere model on $d$-sphere and the Kuramoto model. The novelty of this paper is to derive the kinetic SL equation on the {\it infinite-dimensional} phase space $\bS\fH \times C(\bbt^d)$ together with quantitative fluctuation estimate between one-marginal distribution and one-oscillator distribution in 2-Wasserstein distance. As far as the authors know, the nonlinear transport equation on the infinite-dimensional phase space has not been introduced in kinetic community, although there are some literature 
\cite{A-F, B-D-R1, B-D-R2} on the parabolic type equations on the infinite-dimensional Hilbert spaces in stochastic PDE community. Thus, our work provides the first nontrivial mean-field kinetic model on the infinite-dimensional phase space. For the complete and practical synchronization estimates, we can still use the method of characteristics (or particle paths) on the infinite-dimensional phase space whose existence can be obtained using the classical Cauchy-Lipschitz theory on Banach space, and then by comparing the variations between two particle paths, we can derive a quantitative estimate of the projected diameter of Lagrangian weak solution in $\phi$-space. There are several interesting issues which have not treated in this paper. For example, we have not investigated  the global well-posedness for strong and classical solutions to the proposed kinetic equation and extensions of quantitative and qualitative estimates of the kinetic Kuramoto equation on a finite-dimensional phase space. Moreover, note that the proposed mean-field limit is valid only in a finite-time interval for any initial data. So the extension of this finite-time mean-field limit to a uniform mean-field limit might be also an interesting open problem. We leave these interesting issues for future work.



\newpage

\begin{appendix}


\section{The Cauchy-Lipschitz theory on Banach spaces}\lb{S-CL}

In this appendix, we recall a basic well-posedness theory of ordinary differential equation on Banach spaces. Since the Cauchy-Lipschitz theory for ordinary differential equations is based on Picard's iterations and the Banach fixed point theorem, it uses only completeness and can be stated in Banach spaces (finite dimension
plays no role in that theory). For reader's convenience, we recall the essential facts from Chapter II.1 in \cite{Cartan}. 

Let $\mathfrak X$ be a (separable) Banach space, and let $(a,b)$ be a (nonempty) open interval of $\mathbb R$. Let $V$ be a neighborhood of $(t_0,x_0)$ in $\mathbb R\times\mathfrak X$, and let $F:\,V\mapsto\mathfrak X$ 
be a continuous function. A continuous function $x:\,(a,b)\mapsto\mathfrak X$ is a solution to the follwoing Cauchy problem on the interval $(a,b)$:
\begin{equation} \lb{CauchyPbx0}
\begin{cases}
\displaystyle \dot x(t)=F(t,x(t))\,,\quad t > t_0, \\
\displaystyle  x \Big|_{t = t_0} =x_0,
\end{cases}
\end{equation}
for  $t_0\in (a,b)$ if and only if the map $(a, b)\ni t\mapsto x(t)\in\mathfrak X$ satisfies the following relations:
\[ 
(t,x(t)) \in V \quad \mbox{and} \quad x(t)=x_0+\int_{t_0}^tF(s,x(s))ds, \quad\text{ for all }t \in (a, b). 
\]
In particular, $x\in C^1((a,b); \mathfrak X)$, the differential equation in \eqref{CauchyPbx0} holds for all $t\in(a,b)$, and $x(t_0)=x_0$.

The function $F$ is said to be Lipschitz-continuous in $V$ uniformly in $t$, i.e. there exists $L\ge 0$ such that
\[
\|F(t,x)-F(t,y)\|\le L\|x-y\|
\] 
for all $t\in\mathbb R$ and $x,y\in\mathfrak X$ such that $(t,x)\in V$ and $(t,y)\in V$. This is the formula (1.5.1) in Chapter 2 of \cite{Cartan}.

Assume that $[t_0-T,t_0+T]\times\overline {B(x_0,r)}\subset V$ and that 
\be\lb{BoundHypF}
(t,x)\in [t_0-T,t_0+T]\times\overline {B(x_0,r)}\implies \|F(t,x)\|\le M\,.
\ee
If $F$ is Lipschitz-continuous in $V$ uniformly in $t$, there exists a unique solution 
\[
x:\,[t_0-\tau,t_0+\tau]\to\overline {B(x_0,r)}\subset\mathfrak X
\]
to the problem \eqref{CauchyPbx0}, defined on some time interval such that
\[ \lb{ExistTime}
\tau\ge\min \Big \{ T,\tfrac{r}M \Big \}\,.
\]
This is Corollary 1.7.2 in Chapter 2 of \cite{Cartan}. In particular, this solution is globally defined (for all $t\in \bbr$ and all $y\in\mathfrak X$) if \eqref{BoundHypF} holds with $T=r=+\infty$, i.e. if $F$ is continuous 
and bounded on $V=\bbr \times\mathfrak X$. 

If $V$ is open and $F$ is locally Lipschitz continuous on $V$ uniformly in $t$, and if moreover $x:\,[t_0,t_1)\mapsto\mathfrak X$ is a solution of \eqref{CauchyPbx0} such that $x(t)$ converges to a limit in $\mathfrak X$ as $t\to t_1^-$, 
this solution $x$ can be extended to $[t_0,t_2)$ with $t_2>t_1$. This is the end of Section 1.8 in Chapter 2 of \cite{Cartan}. Next we further assume that $V$ is open in $\bbr \times\mathfrak X$ and that $\partial F/\partial x$ exists and is continuous on $V$. 

Consider the Cauchy problem:
\begin{equation} \label{CauchyPby}
\begin{cases}
\displaystyle \dot x(t)=F(t,x(t))\,,\quad t > t_0, \\
\displaystyle  x \Big|_{t = t_0} =y.
\end{cases}
\end{equation}
For each $r'\in(0,r)$ and each $y\in B(x_0,r')$, the Cauchy problem \eqref{CauchyPby} has a unique solution 
\[
[t_0-\tau',t_0+\tau']\ni t\mapsto X(t,t_0;y)\in\overline{B(x_0,r)}
\]
where $\tau'\ge\min(T,\tfrac{r-r'}M)$. It satisfies the flow identity
\[ \lb{FlowProp}
X(t_2,t_1;X(t_1,t_0;y))=X(t_2,t_0;y)\,,\quad |t_2-t_1|+|t_1-t_0|\le\tfrac{r-r'}M.
\]
and the map $(t_2,t_1;y)\mapsto X(t_2,t_1;y)$ is of class $C^1$ on the domain defined by the inequalities:
\[ |y-x_0|<r \quad \mbox{and} \quad |t_2-t_1|+|t_1-t_0|<\tfrac{r-r'}M. \]
This can be found in Section 3.2 and Theorem 3.4.2 in Chapter 2 of \cite{Cartan}.


\vspace{0.5cm}

\section{The Liouville equation on $(\mathbf S\mathfrak H)^N\times C(\mathbb T^d)^N$}\lb{SLiouv}

In this appendix, we consider the solvability of the Liouville equation on the $N$-oscillator phase space $(\mathbf S\mathfrak H)^N\times C(\mathbb T^d)^N$. \newline

Recall that $\mathfrak H=L^2(\mathbb T^d)$ and that $\mathbf S\mathfrak H=\{\phi\in\mathfrak H\,:\,\|\phi\|_\mathfrak H=1\}$ is the unit sphere in $\mathfrak H$.  Note that the $N$-oscillator ($N$-particle) phase-space of the SL model is $(\mathbf S\mathfrak H)^N\times C(\mathbb T^d)^N$ which is a Polish space. In what follows, we denote by $\mathcal P(\mathcal Z)$ the set 
of Borel probability measures on the Polish space $\mathcal Z$, and we introduce the divergence of a vector field defined on the $N$-oscillator phase space and recall the concept of weak solution of the Liouville equation on the $N$-oscillator phase space.

First, we define the tangent space at any given phase space point. By the analogy with the case of the unit sphere in a finite-dimensional Euclidean space, the tangent space to the unit sphere of $\mathfrak H$ at a point 
$\psi\in\mathbf S\mathfrak H$ is the (closed) hyperplane:
\be\lb{TSH}
T_\psi\mathbf S\mathfrak H:= \Big \{\xi\in\mathfrak H:~\RE\la\psi|\xi\ra=0 \Big \}\,.
\ee
Hence the tangent space at $(\psi_1,\ldots,\psi_N,V_1,\ldots,V_N)$ of the $N$-oscillator phase space is
\be\label{TangSpace}
T_{(\psi_1,\ldots,\psi_N,V_1,\ldots,V_N)}(\bS\fH)^N\times C(\bbt^d)^N=T_{\psi_1}\mathbf S\mathfrak H\times\ldots\times T_{\psi_N}\mathbf S\mathfrak H\times C(\bbt^d)^N\,. 
\ee
A continuous vector field $X$ on $(\bS\fH)^N\times C(\bbt^d)^N$ is a continuous section of the tangent bundle $T((\bS\fH)^N\times C(\bbt^d)^N)$, i.e. a continuous mapping
\[
(\bS\fH)^N\times C(\bbt^d)^N\ni(\psi_1,\ldots,\psi_N,V_1,\ldots,V_N)\mapsto X(\phi_1,\ldots,\phi_N,V_1,\ldots,V_N)\in\mathfrak H^N\times C(\mathbb T^d)^N
\]
such that
\[
X(\psi_1,\ldots,\psi_N,V_1,\ldots,V_N)\in T_{(\psi_1,\ldots,\psi_N,V_1,\ldots,V_N)}(\bS\fH)^N\times C(\bbt^d)^N
\]
for all $(\psi_1,\ldots,\psi_N,V_1,\ldots,V_N)\in(\bS\fH)^N\times C(\bbt^d)^N$. \newline

Next we define a notion of divergence of a vector field on the $N$-oscillator phase space.

\begin{Def}\label{D-B.1} 
Let $X$ be a continuous and bounded ($C_b$-) vector field on $(\bS\fH)^N\times C(\bbt^d)^N$, and let $F_N\in\mathcal P((\bS\fH)^N\times C(\bbt^d)^N)$. Then, the divergence $\Div(F_NX)$ is the continuous linear functional 
on $C^1_b((\bS\fH)^N\times C(\bbt^d)^N)$ defined by the following relation:
\begin{eqnarray*}
&& \la\Div(F_NX),\chi\ra 
\\
&& \hspace{0.5cm} :=-\int_{(\bS\fH)^N\times C(\bbt^d)^N}\la d\chi,X\ra(\psi_1,\ldots,\psi_N,V_1,\ldots,V_N)F_N(d\psi_1\ldots d\psi_N,dV_1\ldots dV_N),
\end{eqnarray*}
for each $\chi\in C^1_b((\bS\fH)^N\times C(\bbt^d)^N)$. 
\end{Def}

Let $t\mapsto X(t,\cdot)$ be a bounded, continuous time-dependent vector field on $(\bS\fH)^N\times C(\bbt^d)^N$. In other words, $X\in C_b( \bbr \times(\bS\fH)^N\times C(\bbt^d)^N;\fH^N\times C(\bbt^d)^N)$, and, 
for each $t\in \bbr$, the map $X(t,\cdot)$ is a vector field on $(\bS\fH)^N\times C(\bbt^d)^N$. We pick $\zeta \in C^\infty(0,+\infty)$ such that
\[
\indc_{[1/2,3/2]}\le \zeta \le\indc_{[1/4,7/4]}\,,
\]
and extend $X$ to $\bbr \times\fH^N\times C(\bbt^d)^N$ by the prescription:
\[
\hat X(\psi_1,\ldots,\psi_N,V_1,\ldots,V_N)=\left(\prod_{j=1}^N\zeta(\|\psi_j\|_\fH)\right)X\left(\tfrac{\psi_1}{\|\psi_1\|_\fH},\ldots,\tfrac{\psi_N}{\|\psi_N\|_\fH},V_1,\ldots,V_N\right)\,.
\]
Assume that $X$ is Lipschitz continuous on $(\bS\fH)^N\times C(\bbt^d)^N$ uniformly in $t\in \bbr$, meaning that
\begin{align*}
\begin{aligned}
& \|X(t,\psi_1,\ldots,\psi_N,V_1,\ldots,V_N)-X(t,\psi'_1,\ldots,\psi'_N,V'_1,\ldots,V'_N)\|_{\fH^N\times C(\bbt^d)^N}
\\
& \hspace{2cm} \le L\sum_{j=1}^N(\|\psi_j-\psi'_j\|_\fH+\|V_j-V'_j\|_{L^\infty})
\end{aligned}
\end{align*}
for some $L\ge 0$. Then the extension $\hat X$ is Lipschitz continuous on $\fH^N\times C(\bbt^d)^N$ uniformly in $t\in \bbr$. By the Cauchy-Lipschitz theorem in Appendix \ref{S-CL}, there exists a global flow $\cS(t,t_0)$ 
on $\fH^N\times C(\bbt^d)^N$ such that $t\mapsto\cS(t,t_0)(\psi_1,\ldots,\psi_N,V_1,\ldots,V_N)$ is the solution of the Cauchy problem
\[
\left\{
\begin{aligned}
{}&\partial_t\cS(t,t_0)(\psi_1,\ldots,\psi_N,V_1,\ldots,V_N)=\hat X(t,\cS(t,t_0)(\psi_1,\ldots,\psi_N,V_1,\ldots,V_N))\,, \quad t > t_0, 
\\
&\cS(t_0,t_0)(\psi_1,\ldots,\psi_N,V_1,\ldots,V_N)=(\psi_1,\ldots,\psi_N,V_1,\ldots,V_N)\,,
\end{aligned}
\right.
\]
for all $(\psi_1,\ldots,\psi_N,V_1,\ldots,V_N)\in\fH^N\times C(\bbt^d)^N$. On the other hand, since
\[
\ba
\psi_1,\ldots,\psi_N\in\bS\fH\implies\hat X(t,\psi_1,\ldots,\psi_N,V_1,\ldots,V_N)=X(t,\psi_1,\ldots,\psi_N,V_1,\ldots,V_N)
\\
\in T_{\psi_1}\mathbf S\mathfrak H\times\ldots\times T_{\psi_N}\mathbf S\mathfrak H\times C(\bbt^d)^N&\,,
\ea
\]
one has
\[
\cS(t,t_0)\left((\bS\fH)^N\times C(\bbt^d)^N\right)\subset(\bS\fH)^N\times C(\bbt^d)^N\qquad\text{ for all }t\in \bbr.
\]
Indeed, decomposing $X=X_s\oplus X_l$ in the product 
\[
T_{(\psi_1,\ldots,\psi_N,V_1,\ldots,V_N)}(\bS\fH)^N\times C(\bbt^d)^N=\left(T_{\psi_1}\mathbf S\mathfrak H\times\ldots\times T_{\psi_N}\mathbf S\mathfrak H\times\{0\}\right)\oplus\left(\{0\}\times C(\bbt^d)^N\right)\,,
\]
one can see that
\[
\tfrac{d}{dt}(\psi_1,\ldots\psi_N)(t)=X_s(t,(\psi_1,\ldots\psi_N)(t))\,,
\]
and
\[
\tfrac{d}{dt}|\psi_j(t)|^2=2\RE(\la X_s(t,(\psi_1,\ldots\psi_N)(t))|(0,\ldots,\psi_j(t),\ldots,0)\ra=0\,,\quad j \in [N]\,.
\]
In particular, the restricted flow $\cS(t,t_0;\cdot)\rstr_{(\bS\fH)^N\times C(\bbt^d)^N}$ depends only on the vector field $X$ on $(\bS\fH)^N\times C(\bbt^d)^N$ and not on its extension $\hat X$ to the ambient 
linear space $\fH^N\times C(\bbt^d)^N$.

\begin{Lem}\lb{L-B.1}
Let $t\mapsto X(t,\cdot)$ be a continuous and bounded, time-dependent vector field on $(\bS\fH)^N\times C(\bbt^d)^N$ that is Lipschitz continuous on $\fH^N\times C(\bbt^d)^N$ uniformly in $t\in \bbr$, and let
$\cS(t,t_0)$ be the flow generated by the vector field $X$ on $(\bS\fH)^N\times C(\bbt^d)^N$. For each initial datum $F^{0}\in\cP((\bS\fH)^N\times C(\bbt^d)^N)$, we set
\be\lb{CharFla}
F(t,\cdot):=\cS(t,0)\#F^{0}\,.
\ee
Then $F$ is a weak solution of the Liouville equation on $(\bS\fH)^N\times C(\bbt^d)^N$:
\begin{equation} \lb{Liouville}
\begin{cases}
\displaystyle \partial_tF+\Div(FX)=0, \quad t > 0, \\
\displaystyle F\rstr_{t=0}=F^{0}.
\end{cases}
\end{equation}
\end{Lem}
\begin{proof}
Let $\chi\in C_b^1((\mathbf S\mathfrak H)^N\times C(\mathbb T^d)^N)$. Then, we have
\begin{align*}
\begin{aligned}
& \frac{d}{dt}\int_{(\bS\fH)^N\times C(\mathbb T^d)^N}\chi(\phi_1\ldots,\phi_N,V_1,\ldots,N)F(t,d\phi_1\ldots d\phi_NdV_1\ldots dV_N) 
\\
& \hspace{0.5cm} =\!\int_{(\mathbf S\mathfrak H)^N\times C(\mathbb T^d)^N}\!\partial_t\chi(\cS(t,\!0)(\phi_1^0,\ldots,\phi_N^0,V_1^0,\ldots,V_N^0))F^0(d\phi_1^0\ldots d\phi_N^0dV_1^0\ldots dV_N^0)
\\
& \hspace{0.5cm} =\!\int_{(\mathbf S\mathfrak H)^N\times C(\mathbb T^d)^N}\!\la d\chi,\!X(t,\cdot)\ra(\cS(t,\!0)(\phi_1^0,\ldots,\phi_N^0,V_1^0,\ldots,V_N^0))F^0(d\phi_1^0\ldots d\phi_N^0dV_1^0\ldots dV_N^0)
\\
& \hspace{0.5cm} =\!\int_{(\mathbf S\mathfrak H)^N\times C(\mathbb T^d)^N}\!\la d\chi,X(t,\cdot)\ra(\phi_1,\ldots,\phi_N,V_1,\ldots,V_N)F(t,d\phi_1\ldots d\phi_NdV_1\ldots dV_N)
\\
& \hspace{0.5cm} =\!-\la\Div(F(t,\cdot)X(t,\cdot)),\chi\ra.
\end{aligned}
\end{align*}
This gives the desired result.
\end{proof}
\begin{Rmk} \label{RB.1}
Solutions of \eqref{Liouville} of the form \eqref{CharFla} will be referred to as ``Lagrangian weak solutions''.
\end{Rmk}

\end{appendix}

\bigskip
\noindent
\textbf{Data availability statement.}
We do not analyse or generate any datasets, because our work proceeds within a theoretical and mathematical approach. One can obtain the relevant materials from the references below.

\smallskip
\noindent
\textbf{Ethics declarations:}

\noindent
\textbf{Competing interests.} The authors declare no competing interests.


\bibliographystyle{amsplain}


\end{document}